\documentclass[10pt]{amsart}
\usepackage{amsfonts,amsmath,amsthm,amssymb, color}
\usepackage[all]{xy}
\usepackage[
colorlinks=true, linkcolor=blue, citecolor=blue]{hyperref}

\numberwithin{equation}{section}

\begin{document}

\newtheorem{theorem}{Theorem}[section]
\newtheorem{lemma}[theorem]{Lemma}
\newtheorem{proposition}[theorem]{Proposition}
\newtheorem{corollary}[theorem]{Corollary}

\theoremstyle{definition}
\newtheorem{definition}[theorem]{Definition}
\newtheorem{example}[theorem]{Example}

\theoremstyle{remark}
\newtheorem{remark}[theorem]{Remark}

\newenvironment{magarray}[1]
{\renewcommand\arraystretch{#1}}
{\renewcommand\arraystretch{1}}

\newcommand{\mapor}[1]{\smash{\mathop{\longrightarrow}\limits^{#1}}}
\newcommand{\mapin}[1]{\smash{\mathop{\hookrightarrow}\limits^{#1}}}
\newcommand{\mapver}[1]{\Big\downarrow
\rlap{$\vcenter{\hbox{$\scriptstyle#1$}}$}}
\newcommand{\liminv}{\smash{\mathop{\lim}\limits_{\leftarrow}\,}}

\newcommand{\Set}{\mathbf{Set}}
\newcommand{\Art}{\mathbf{Art}}
\newcommand{\solose}{\Rightarrow}

\renewcommand{\bar}{\overline}
\newcommand{\de}{\partial}
\newcommand{\debar}{{\overline{\partial}}}
\newcommand{\per}{\!\cdot\!}
\newcommand{\Oh}{\mathcal{O}}
\newcommand{\sA}{\mathcal{A}}
\newcommand{\sB}{\mathcal{B}}
\newcommand{\sC}{\mathcal{C}}
\newcommand{\sF}{\mathcal{F}}
\newcommand{\sG}{\mathcal{G}}
\newcommand{\sH}{\mathcal{H}}
\newcommand{\sI}{\mathcal{I}}
\newcommand{\sL}{\mathcal{L}}
\newcommand{\sM}{\mathcal{M}}
\newcommand{\sP}{\mathcal{P}}
\newcommand{\sU}{\mathcal{U}}
\newcommand{\sV}{\mathcal{V}}
\newcommand{\sX}{\mathcal{X}}
\newcommand{\sY}{\mathcal{Y}}

\newcommand{\Aut}{\operatorname{Aut}}
\newcommand{\Mor}{\operatorname{Mor}}
\newcommand{\Def}{\operatorname{Def}}
\newcommand{\Hom}{\operatorname{Hom}}
\newcommand{\Alt}{\operatorname{Alt}}
\newcommand{\Tot}{\operatorname{Tot}}
\newcommand{\HOM}{\operatorname{\mathcal H}\!\!om}
\newcommand{\DER}{\operatorname{\mathcal D}\!er}
\newcommand{\Spec}{\operatorname{Spec}}
\newcommand{\Der}{\operatorname{Der}}
\newcommand{\Tor}{{\operatorname{Tor}}}
\newcommand{\Ext}{\operatorname{Ext}}
\newcommand{\Sym}{\operatorname{Sym}}
\newcommand{\End}{{\operatorname{End}}}
\newcommand{\END}{\operatorname{\mathcal E}\!\!nd}
\newcommand{\Image}{\operatorname{Im}}
\newcommand{\coker}{\operatorname{coker}}
\newcommand{\Id}{\operatorname{Id}}

\renewcommand{\Hat}[1]{\widehat{#1}}
\newcommand{\dual}{^{\vee}}
\newcommand{\desude}[2]{\dfrac{\de #1}{\de #2}}

\newcommand{\A}{\mathbb{A}}
\newcommand{\N}{\mathbb{N}}
\newcommand{\R}{\mathbb{R}}
\newcommand{\Z}{\mathbb{Z}}
\renewcommand{\H}{\mathbb{H}}
\newcommand{\C}{\mathbb{C}}
\newcommand{\proj}{\mathbb{P}}
\newcommand{\K}{\mathbb{K}\,}


\newcommand{\contr}{{\mspace{1mu}\lrcorner\mspace{1.5mu}}}
\newcommand{\wedgebar}{\;{\scriptstyle\mathchar'26\mkern-10.5mu\wedge}\;}

\newcommand{\bi}{\boldsymbol{i}}
\newcommand{\bl}{\boldsymbol{l}}

\newcommand{\MC}{\operatorname{MC}}
\newcommand{\Coder}{{\operatorname{Coder}}}

\newcommand{\QUI}{\bigskip\bigskip\textbf{********** Segnaposto ********}\bigskip\bigskip}

\title{On some formality criteria for DG-Lie algebras}
\date{\today}

\author{Marco Manetti}
\address{\newline
Universit\`a degli studi di Roma ``La Sapienza'',\hfill\newline
Dipartimento di Matematica \lq\lq Guido
Castelnuovo\rq\rq,\hfill\newline
P.le Aldo Moro 5,
I-00185 Roma, Italy.}
\email{manetti@mat.uniroma1.it}
\urladdr{www.mat.uniroma1.it/people/manetti/}

\renewcommand{\subjclassname}{%
\textup{2010} Mathematics Subject Classification}


\begin{abstract}
We give some formality criteria for a differential graded Lie algebra to be formal. 
For instance, we show that
a DG-Lie algebra $L$ is formal if and only if the natural spectral sequence computing 
the Chevalley--Eilenberg cohomology  $H^*_{CE}(L,L)$  degenerates at $E_2$. 
\end{abstract}

\maketitle

\section{Introduction}

The notion of formality of a differential graded commutative algebra has been quite familiar 
in mathematics since the works by Deligne, Griffiths, Morgan, Sullivan \cite{DGMS}, where it is proved that
the de Rham algebra of a compact K\"{a}hler manifold $X$ is formal and therefore its homotopy class, controlling the real homotopy type of $X$, is uniquely determined by the cohomology algebra $H^*(X,\mathbb{R})$.

Similarly, the notion of formality of a differential graded Lie algebra  has received a great attention after the papers of Goldman, Millson \cite{GoMil1} and Kontsevich \cite{K}. 
In  \cite{GoMil1} the authors realize that  the same approach of \cite{DGMS} can be used to prove the 
formality of the differential graded Lie algebra of differential forms with values in certain flat bundles of Lie algebras; as a   consequence of this fact they proved that the moduli space of certain representations of the fundamental  group of a   compact K\"{a}hler manifold has at most quadratic singularities.

In the paper \cite{K}  Kontsevich proved that, when $A$ is the algebra of smooth functions on a differentiable 
manifold, then the natural DG-Lie algebra structure on the 
Hochschild cohomology complex of $A$ with coefficients in $A$ is formal, and then  proving that 
every finite dimensional Poisson manifolds admits a canonical deformation quantization.

Recall that a DG-Lie algebra $L$ is formal if there exists a pair of quasi-isomorphisms of DG-Lie algebras
\[ L\xleftarrow{\qquad}M\xrightarrow{\qquad}H\]
with $H$ having trivial differential. A   
DG-Lie algebra $L$ is called homotopy abelian if there exists a pair of quasi-isomorphisms of DG-Lie algebras

\[ L\xleftarrow{\qquad}M\xrightarrow{\qquad}H\]
with $H$ having trivial bracket and trivial differential. Thus,  a DG-Lie algebra $L$ is homotopy abelian if and only if it is formal and the cohomology graded Lie algebra $H^*(L)$ is abelian.

In view of the general principle that, in characteristic 0 every deformation problem is controlled by a DG-lie algebras, with quasi-isomorphic DG-Lie algebras giving the same deformation problem \cite{GoMil1}, 
both the notions of formality  and homotopy abelianity play a central role in deformation theory. 
Fortunately, in literature there exist several general and very useful criteria for  homotopy abelianity. 
For example, it is known (see e.g. \cite{semireg2011} and references therein) that for a morphism 
of DG-Lie algebras 
 $f\colon L\to M$ we have:

\begin{enumerate}

\item if $L$ is homotopy abelian and $f\colon H^*(L)\to H^*(M)$ is surjective, then also $M$ is homotopy abelian;

\item if $M$ is homotopy abelian and $f\colon H^*(L)\to H^*(M)$ is injective, then also $L$ is homotopy abelian.
\end{enumerate}

The above results, which are completely symmetric in their proofs, appear quite different in 
their applications, being the latter used in almost all the algebraic proofs of generalized 
Bogomolov--Tian--Todorov theorems \cite{IaconoDP,KKP}, as well in deformation theory of holomorphic Poisson manifolds and coisotropic submanifolds \cite{BaMacoisotropic,FMpoisson}.

The initial motivation for this paper  was to seek for an analog of 
the above item (2) when the notion of homotopy abelianity is replaced with the notion of formality; 
it is easily verified that if $M$ is formal, then the injectivity of $f\colon H^*(L)\to H^*(M)$ 
is not sufficient to ensure the formality of $L$. 

In our proposed extension  of item (2) for formality (Theorem~\ref{thm.formalitytransfer}),
the cohomology graded Lie algebras $H^*(L)$ and $H^*(M)$ are replaced with suitable Chevalley--Eilenberg cohomology groups. In particular we shall prove that if $f\colon L\to M$ is a morphism of differential graded Lie algebras, with $M$ formal and   
$f\colon H^2_{CE}(H^*(L),H^*(L))\to H^2_{CE}(H^*(L),H^*(M))$ injective, then also $L$ is formal.
 
In doing this we have been deeply inspired by the papers \cite{kaledin,lunts}, where it 
is explained  what is  the ``right'' obstruction to formality of a DG-algebra and by 
\cite{bandierakapra}, 
where it is proved that  homotopy abelianity is equivalent to   
the degeneration at $E_1$ of the natural spectral sequence computing the 
Chevalley--Eilenberg cohomology.

We expect that the results and the techniques presented here could find application and to be source of future research also in rational homotopy theory, especially in the description of 
the rational homotopy of non-connected spaces via Maurer--Cartan elements modulo the gauge action in the context of $L_{\infty}$-algebras \cite{berglund,bumu,laza}.

The paper is organized as follows: in Section~\ref{sec.CEspectral} we introduce the Chevalley--Eilenberg complex of a differential graded Lie algebra as the natural generalization of the classical Chevalley--Eilenberg complex of a Lie algebra \cite{CE}; here the choice of signs of the differential is purely teleological and 
made by taking into account the 
sign convention used in  the definition of d\'ecalage maps given in 
Section~\ref{sec.formalityDGLA}. Such a complex admits a natural filtration giving a cohomology spectral sequence.

Since almost all the proofs of this paper 
require a good knowledge of $L_{\infty}[1]$-algebras and $L_{\infty}$-morphisms, 
for the benefit of the readers which are not familiar with these notions, in Section~\ref{sec.mainresults} 
we state the main results of the paper about formality of DG-Lie algebras. 
These results will be proved in Section~\ref{sec.formalityDGLA} as particular cases of some more general results 
concerning the formality of $L_{\infty}[1]$-algebras. 
Among the applications of these results we give a proof of the fact that for every graded vector space $V$ the graded Lie algebra $\Hom^*_{\K}(V,V)$ is intrinsically formal.

In Section~\ref{sec.review} we give a short review of the definition of $L_{\infty}[1]$-algebras, $L_{\infty}$-morphisms and Nijenhuis--Richardson bracket. In Section~\ref{sec.homotopyinvariance} we define the 
Chevalley--Eilenberg spectral sequence of an   $L_{\infty}[1]$-algebra and we prove that, 
for every $r>0$, its page $E_r$ is homotopy invariant. 
Finally in Section~\ref{sec.criteriaLinfinity} we prove the formality criteria for $L_{\infty}[1]$-algebras.

In the last section  we show how deformation theory can be used for constructing 
simple examples of non formal differential graded Lie algebras.

\bigskip
\section{The Chevalley--Eilenberg spectral sequence}
\label{sec.CEspectral}

Throughout this paper every vector space, tensor product, Lie algebra etc. is considered over a fixed field $\K$ 
of characteristic 0. By a DG-vector space we shall mean a $\Z$-graded vector space equipped with a 
differential of degree $+1$; a DG-Lie algebra is a Lie object in the category of 
DG-vector spaces. Given a homogeneous vector $v$ on a graded vector space, its degree will be denoted 
by $\bar{v}$.

\medskip
Let $L=(L,d,[-,-])$ be a differential graded Lie algebra and 
let $M$ be an $L$-module. 
This means that $M=(M,d)$ is a DG-vector space and it is given a 
morphism of DG-vector spaces
\[ [-,-]\colon M\otimes L\to M\]
such that 
\[[m,[x,y]]=[[m,x],y]-(-1)^{\bar{x}\;\bar{y}}[[m,y],x]\;.\] 
For instance, 
if  $f\colon L\to M$ is a morphism of DG-Lie algebras, then $M$ is an $L$-module via the adjoint representation
$[m,x]=[m,f(x)]$.

We shall denote by $H^*(L)$ and $H^*(M)$ the cohomology of the DG-vector spaces $(L,d)$ and 
$(M,d)$, respectively.
For every integer $p\ge 0$ let's consider the DG-vector space
\[ CE(L,M)^{p,*}=\Hom^*_{\K}(L^{\wedge p},M),\]
equipped with the natural  differential $\bar{\delta}\colon CE(L,M)^{p,q}\to CE(L,M)^{p,q+1}$, namely \[ (\bar{\delta}\phi)(x_1,\ldots,x_p)=
d(\phi(x_1,\ldots,x_p))-\sum_{i=1}^p(-1)^{\bar{\phi}+\bar{x_1}+\cdots+\bar{x_{i-1}}}
\phi(x_1,\ldots, dx_i,\ldots,x_p)\;,\]
where every element of $CE(L,M)^{p,*}$ is interpreted as a $p$-linear graded skewsymmetric map
$L\times\cdots\times L\to M$. 
As usual we assume that $L^{\wedge 0}=\K$ and then $CE(L,M)^{0,*}=M$.

Following \cite[pag. 94]{Ja}, the Chevalley--Eilenberg complex of $L$ with coefficients in 
$M$  is the complex of DG-vector spaces:   
\[ CE(L,M)\colon\quad 0\to CE(L,M)^{0,*}\xrightarrow{\,\delta\,}CE(L,M)^{1,*}\xrightarrow{\,\delta\,}
CE(L,M)^{2,*}\to\cdots\;,\]
i.e., the complex
\[ CE(L,M)\colon\quad 0\to M\xrightarrow{\,\delta\,}\Hom_{\K}^*(L,M)\xrightarrow{\,\delta\,}\Hom_{\K}^*(L^{\wedge 2},M)\to
\cdots\]
where the differential $\delta$ is defined in the following way:
\begin{enumerate}
\item for  $m\in M$ we have $(\delta m)(x)=(-1)^{\bar{m}}[m,x]$;

\item  for $\phi\in \Hom_{\K}^*(L,M)$ we have 
\[(\delta\phi)(x,y)=(-1)^{\bar{\phi}+1}\left([\phi(x),y]-
(-1)^{\bar{x}\;\bar{y}}[\phi(y),x]-\phi([x,y])\right);\]

\item  for $p\ge 2$ and $\phi\in \Hom_{\K}^*(L^{\wedge p-1},M)$  we have:
\[ \begin{split}
&(\delta \phi)(x_1,\ldots,x_p)=\\
&=(-1)^{\bar{\phi}+p-1}\left(\!\sum_{i}\chi_i\, [\phi(x_1,\ldots,\widehat{x_i},\ldots,x_p),x_i]
-\!\sum_{i<j}\chi_{i,j}\, \phi(x_1,\ldots,\widehat{x_i},\ldots,\widehat{x_j},\ldots,x_p,[x_i,x_j])\right)
\end{split}\] 
where the  $\chi_i,\chi_{i,j}\in \{\pm 1\}$ are the antisymmetric Koszul signs; when 
$x_1\wedge\cdots\wedge x_p\not=0$ they are determined by the following equalities in 
$L^{\wedge p}$:
\[ x_1\wedge\cdots\wedge x_p=\chi_i\;
x_1\wedge\cdots\widehat{x_i}\cdots\wedge x_p\wedge x_i=\chi_{i,j}\;
 x_1\wedge\cdots\widehat{x_i}\cdots\widehat{x_j}\cdots\wedge x_p\wedge x_i\wedge x_j\;.\]
\end{enumerate}

The proof that  $\delta^2=0$ 
may be easily reduced to the case $M=L$: in fact, denoting by 
$K=L\oplus M$ the trivial extension of the DG-Lie algebra $L$ by the $L$-module $M$, there exists a 
natural  sequence of 
embeddings of   DG-vector spaces $\Hom_{\K}^*(L^{\wedge p},M)\subset \Hom_{\K}^*(K^{\wedge p},K)$ 
commuting with the operators $\delta$.  Now, under  the assumption $M=L$ the proof becomes 
tedious but completely straightforward, cf. also \cite{CalRos,CE,Pen}. Alternatively one can 
observe that when $\phi,x_1,\ldots,x_p$ have even degree the above definition of $\delta$ is, 
up to sign,  the same  
of \cite{Ja}; then one
can use the standard trick (cf. \cite{bandieraderived,voronov2}) of 
taking the tensor products of $L,M$ with a suitable Grassmann algebra in order to reduce the verification of 
$(\delta^2\phi)(x_1,\ldots,x_p)=0$ to the case where $\phi,x_1,\ldots,x_p$ have even degree.
A more conceptual description of the Chevalley--Eilenberg complex  will be given later as the und\'ecalage 
of the DG-vector space of coderivations of a differential symmetric coalgebra,  cf. also \cite{lazarev,Pen}.

Notice that for $\phi\in \Hom^*_{\K}(L,M)$, the condition $\delta\phi=0$ can be written as 
\[ \phi([x,y])=[\phi(x),y]-
(-1)^{\bar{x}\;\bar{y}}[\phi(y),x]=[\phi(x),y]+(-1)^{\bar{x}\;\bar{\phi}}[x,\phi(y)]\]
and then the kernel of $\Hom_{\K}^*(L,M)\xrightarrow{\delta}\Hom_{\K}^*(L^{\wedge 2},M)$ is the space of 
derivations $\phi\colon L\to M$. On the other side the image of 
$\delta\colon M\to \Hom_{\K}^*(L,M)$ is, by definition, the space of inner derivations and then 
\[ H^1(CE(L,M),\delta)=\frac{\{\text{derivations}\;  L\to M\}}{\{\text{inner derivations}\}}\;.\]

Similarly it is proved that $\bar{\delta}\delta+\delta \bar{\delta}=0$, giving  a double complex structure
$(CE(L,M),\delta,\bar{\delta})$.

\begin{definition}
The Chevalley--Eilenberg 
cohomology $H^*_{CE}(L,M)$ of the differential graded Lie algebra $L$ 
with coefficients in the $L$-module $M$ is the cohomology of the total complex 
$\operatorname{Tot}^{\prod}(CE(L,M),\delta,\bar{\delta})$.   
\end{definition}

In other words,  $H^*_{CE}(L,M)$ is the cohomology of the complex 
$\cdots\to A^i\xrightarrow{\delta+\bar{\delta}}A^{i+1}\to\cdots$, where
\[ A^n=\prod_{p+q=n}\Hom^{q}_{\K}(L^{\wedge p},M)\;.\]

The Chevalley--Eilenberg complex carries the 
natural, decreasing, exhaustive and complete filtration 
\[F^pCE(L,M)=\Hom^*_{\K}\left(\bigoplus_{i\ge p}{\bigwedge}^{i}L,M\right),\qquad p\ge 0\;.\] 
We shall denote by $(E(L,M)^{p,q}_r, d_r)$ the associated (Chevalley--Eilenberg) cohomology spectral sequence.

\begin{example} 
If $L,M$ have trivial differentials, then $H^*(L)=L$, 
$H^*(M)=M$ and therefore $E(L,M)^{p,q}_0=E(L,M)^{p,q}_1$. Moreover the 
spectral sequence degenerates  at $E_2$ (i.e., $d_r=0$ for every $r\ge 2$) and 
\[ H^i_{CE}(L,M)=\prod_{p\ge 0}E(L,M)^{p,i-p}_2,\]
where
\[E(L,M)^{p,i-p}_2=\frac{\ker(\Hom^{i-p}_{\K}(L^{\wedge p},M)\xrightarrow{\delta}
\Hom^{i-p}_{\K}(L^{\wedge p+1},M))}{\delta\Hom^{i-p}_{\K}(L^{\wedge p-1},M)}\;.\]
\end{example}

In general, since 
\[E(L,M)^{p,*}_0=\frac{F^pCE(L,M)}{F^{p+1}CE(L,M)}=\Hom^*_{\K}(L^{\wedge p},M)\] 
and the field $\K$ is assumed to be of characteristic $0$, we have $H^*(L^{\wedge p})=H^*(L)^{\wedge p}$ (see e.g. 
\cite[pag. 280]{Qui})
and then 
\begin{equation}\label{equ.e1term}
\begin{split}
E(L,M)^{p,q}_1&=H^q(\Hom^*_{\K}(L^{\wedge p},M),\bar{\delta})=
\Hom^q_{\K}(H^*(L)^{\wedge p},H^*(M))\\
&=E(H^*(L),H^*(M))^{p,q}_1\;.\end{split}
\end{equation}
 
The differential $d_1\colon E(L,M)_1^{p,q}\to E(L,M)_1^{p+1,q}$ depends only on the graded 
Lie algebra $H^*(L)$ and its module $H^*(M)$, giving 
\[   E(L,M)^{p,*}_2=E(H^*(L),H^*(M))^{p,*}_2=H^p(CE(H^*(L),H^*(M)),\delta)\]
and therefore  
\[   E(L,M)^{1,*}_2=E(H^*(L),H^*(M))^{1,*}_2=
\frac{\{\text{derivations}\;  H^*(L)\to H^*(M)\}}{\{\text{inner derivations}\}}\;.\]

\bigskip
\section{Statement of the main results}
\label{sec.mainresults}

For the clarity of exposition, we list here the main results proved in this paper; the proofs rely on the theory of $L_{\infty}[1]$-algebras and will be postponed to next sections.

\begin{definition}[Euler class]\label{def.eulerclass} The \emph{Euler class} of a
morphism of differential graded Lie algebras $f\colon L\to M$
is the element $e_f\in E(L,M)^{1,0}_2=E(H^*(L),H^*(M))^{1,0}_2$ corresponding 
to the Euler derivation  
\[ e_f\colon H^*(L)\to H^*(M),\qquad e_f(x)=\bar{x}\, f(x)\;.\]
The Euler class of a DG-Lie algebra $L$ is defined as the Euler class of the identity on $L$.
\end{definition}

\begin{lemma}\label{lem.degenerazioneparziale} 
Let $e\in E(L,L)^{1,0}_2$ be the Euler class of a differential graded Lie algebra $L$. If 
$d_2(e)=\cdots=d_{r-1}(e)=0$ for some $r>2$, then 
$d_k=0$ for every $2\le k<r$; in particular   
$E(L,L)^{p,q}_r=E(L,L)^{p,q}_2$.
\end{lemma}

Every morphism of differential graded Lie algebras $f\colon L\to M$ induces by composition 
two natural morphisms of  double complexes 
\[ \xymatrix{CE(L,L)\ar[r]^{f_*\;}&CE(L,M)&CE(M,M)\ar[l]_{\;f^*}}\]
and then also two morphisms of spectral sequences
\begin{equation}\label{equ.quasiisoateuno} 
\xymatrix{E(L,L)^{p,q}_r\ar[r]^{f_*\;}&E(L,M)^{p,q}_r&E(M,M)^{p,q}_r\ar[l]_{\;f^*}}\;.
\end{equation}
preserving Euler classes.
If $f$ is a quasi-isomorphism, then by \eqref{equ.e1term}  
the   maps in \eqref{equ.quasiisoateuno} are isomorphisms for  $r\ge 1$.  
Therefore, \emph{the truncation at $r\ge 1$ of the spectral sequence $E(L,L)^{p,q}_r$  and the Euler class
are homotopy invariants of  $L$}.

\begin{theorem}[Formality criterion]\label{thm.formalitycriteria} 
Let $(E(L,L)^{p,q}_r,d_r)$
be the Chevalley--Eilenberg spectral sequence of  a differential graded Lie algebra $L$. 
Then the following conditions are equivalent:

\begin{enumerate}

\item\label{it1.thm.formalitycriteria} $L$ is formal;

\item the spectral sequence $E(L,L)^{p,q}_r$ degenerates at $E_2$;

\item\label{it3.thm.formalitycriteria}  denoting by 
\[ e\in E(L,L)^{1,0}_2=\frac{\;\Der^0_{\K}(H^*(L),H^*(L))\;}{
\{[x,-]\mid x\in H^0(L)\}},\qquad e(x)=\bar{x}\cdot x,\]
the Euler class of $L$, we have $d_r(e)=0\in E(L,L)^{r+1,1-r}_r$ for every $r\ge 2$;
\end{enumerate}
\end{theorem}

According to Lemma~\ref{lem.degenerazioneparziale} the above item \eqref{it3.thm.formalitycriteria} 
makes sense. By the above considerations about the homotopy invariance of the Chevalley--Eilenberg spectral sequence and Euler classes,  the only non trivial implication  is 
$(\ref{it3.thm.formalitycriteria}\Rightarrow \ref{it1.thm.formalitycriteria})$.

\begin{theorem}[Formality transfer]\label{thm.formalitytransfer} 
Let $f\colon L\to M$ be a morphism of 
differential graded Lie algebras. Assume that 
\begin{enumerate}

\item\label{it1.thm.formalitytransfer}  $M$ is formal;

\item\label{it2.thm.formalitytransfer}   for every $p\ge 3$ the map 
\[ f\colon E(H^*(L),H^*(L))^{p,2-p}_2\to E(H^*(L),H^*(M))^{p,2-p}_2\]
is injective.
\end{enumerate}
Then also $L$ is formal.
\end{theorem}

As we have already pointed out, the above Item~\eqref{it2.thm.formalitytransfer} holds whenever 
the natural map 
\[ f\colon H^2_{CE}(H^*(L),H^*(L))\to H^2_{CE}(H^*(L),H^*(M))\]
is injective.

\begin{corollary}\label{cor.prodottoLie}
Let $L,M$ be  differential graded Lie algebras. Then $L\times M$ is formal if and only if both $L$ and $M$ are formal.
\end{corollary}

\begin{proof}  Immediate consequence of Theorem~\ref{thm.formalitytransfer}, since 
$L$ (resp.: $M$) is a direct summand of the $L$-module (resp.: $M$-module) $L\times M$.
\end{proof}

The next corollary is the Lie analog of a remarkable result by Sullivan, Halperin and Stasheff 
\cite[Cor. 6.9]{halsta}. Recall that for every differential graded Lie algebra $L$ and every graded commutative algebra $A$, the tensor product $L\otimes A$ carries a natural structure of DG-Lie algebra, where the bracket and the differential are determined by the formulas
\[ [l_1\otimes a_1,l_2\otimes a_2]=(-1)^{\bar{a_1}\,\bar{l_2}}[l_1,l_2]\otimes a_1a_2,\qquad
d(l\otimes a)=dl\otimes a+(-1)^{\bar{l}}l\otimes da\;.\]

\begin{corollary}\label{cor.tensorformality}
Let $L$ be a differential graded Lie algebra and let $A$ be a unitary differential graded commutative $\K$-algebra. 
If $H^*(A)\not=0$ and 
$L\otimes A$ is formal, then also $L$ is formal.
\end{corollary}

\begin{proof}  Let's denote by $d\colon A^i\to A^{i+1}$ the differential of $A$, then $d(1)=0$ and the assumption $H^*(A)\not=0$ implies that the cohomology class of $1$ is non trivial in $H^0(A)$: in fact, if $1=da$ for some $a\in A^{-1}$, then for every $b\in A$ such that $d(b)=0$ we have 
$d(ab)=b$.
Thus the morphism $\K\to A$, $\alpha\mapsto \alpha 1$, is injective in cohomology and therefore there exists a direct sum decomposition $A=\K\oplus B$ with 
$d(B)\subseteq B$.

Now the proof follows from Theorem~\ref{thm.formalitytransfer}, since 
$L$ is a direct summand of the $L$-module $L\otimes A=L\oplus (L\otimes B)$.
\end{proof}

Notice that, if $A$ is not formal as differential graded commutative algebra, then the formality of $L$ does not imply the formality of $L\otimes A$: for an example see e.g. 
\cite[Prop. 7.2]{indam}. Similarly the formality of $L$ and $L\otimes A$ does not imply the formality of $A$.

\begin{definition}[\cite{hinich}] 
A graded Lie algebra $H$ is \emph{intrinsically formal} if every differential graded Lie algebra $L$ such that  
$H^*(L)\cong H$ is formal.
\end{definition} 

Putting $M=0$ in Theorem~\ref{thm.formalitytransfer} we recover the well known fact \cite{hinich,kaledin} that a graded Lie algebra $L$ with $E(L,L)^{p,2-p}_2=0$ for every $p\ge 3$ is intrinsically formal. After Theorem~\ref{thm.formalitycriteria},
another sufficient condition for intrinsic formality is given by the vanishing of the Euler class.

\begin{corollary} 
For every graded Lie algebra $M$ and every $h\in M^0$ the graded Lie subalgebra
\[ H=\bigoplus_i\{ x\in M^i\mid [h,x]=ix\}\]
is intrinsically formal.  
\end{corollary}

\begin{proof} Notice first that $h\in H^0$ and then the Euler derivation $e=[h,-]\colon H\to H$ is an inner derivation.
Let $L$ be a differential graded Lie algebra with $H^*(L)=H$, then 
\[E(L,L)^{1,0}_2=E(H,H)^{1,0}_2=\frac{\;\Der^0_{\K}(H,H)\;}{
\{[x,-]\mid x\in H^0\}},\]
and therefore the Euler class  is trivial in $E(L,L)^{1,0}_2$.
\end{proof}

\begin{example} For every graded vector space $V$, the graded Lie algebras 
$\Hom^*_{\K}(V,V)$, $\Hom^{\ge 0}_{\K}(V,V)$ and $\Hom^{\le 0}_{\K}(V,V)$   
are intrinsically formal. In fact, denoting by 
\[ h\in \Hom^0_{\K}(V,V), \qquad h(v)=\bar{v}\, v,\]
we have 
\[ [h,f]=e(f)= \bar{f}\, f,\qquad \text{for every}\quad f\in \Hom^*_{\K}(V,V)\;.\]
\end{example}

\begin{example} For every graded commutative algebra $A$, the graded Lie algebra $\Der^*_{\K}(A,A)$ is intrinsically formal. In fact, denoting by 
\[ h\in \Der^0_{\K}(A,A), \qquad h(v)=\bar{v}\, v,\]
we have 
\[ [h,f]=e(f)= \bar{f}\, f,\qquad \text{for every}\quad f\in \Der^*_{\K}(A,A)\;.\]
The same conclusion applies to every graded Lie subalgebra of $\Hom^*_{\K}(A,A)$ containing the derivation 
$h$, e.g. the algebra of differential operators.
\end{example}

\bigskip
\section{Review of $L_{\infty}[1]$-algebras and Nijenhuis--Richardson bracket}
\label{sec.review}

Given a graded vector space $V$, the  symmetric 
coalgebra generated by $V$ is the graded vector space  
$S^{c}(V)=\oplus_{n\ge 0}V^{\odot n}$  equipped with the  coproduct 
\[ \Delta(1)=1\otimes 1,\qquad \Delta(v)=1\otimes v+v\otimes 1,\]
and more generally
\[ \Delta\colon S^{c}(V)\to S^{c}(V)\otimes S^{c}(V),\]
\[ \Delta(v_1\odot\cdots\odot v_n)=
\sum_{k=0}^{n}\sum_{\sigma\in S(n,n-k)}\epsilon(\sigma) 
(v_{\sigma(1)}\odot\cdots\odot  v_{\sigma(k)})\otimes 
(v_{\sigma(k+1)}\odot\cdots\odot  v_{\sigma(n)})\]
where $S(k,n-k)$ is the set of permutations of $\{1,\ldots,n\}$ such that 
\[ \sigma(1)<\cdots<\sigma(k),\qquad \sigma(k+1)<\cdots<\sigma(n),\]
and $\epsilon(\sigma)$ is the Koszul sign.
For every positive integer $m$ the subspace
$\bigoplus_{0\le n\le m}V^{\odot n}$ is 
a graded subcoalgebra of 
$S^c(V)$.

Throughout all this paper we shall use in force the following notation: whenever $i,j\ge 0$ and
$f\colon S^c(V)\to S^c(W)$ is a linear map, we shall denote by 
$f^i_j\colon V^{\odot j}\to W^{\odot i}$ the composite map 
\[ V^{\odot j}\xrightarrow{\text{inclusion}}
S^c(V)\xrightarrow{\; f\;} S^c(W)\xrightarrow{\text{projection}}W^{\odot i}\;.\]
The composition of $f$ with the projection $S^c(W)\to W$ is called the \emph{corestriction of $f$}; equivalently the corestriction of $f$ is  the linear map $\sum_{j\ge 0} f^1_j$.

The projection $S^c(V)\to V^{\odot 0}=\K$ is a counity, while the inclusion $\K=V^{\odot 0}\to S^c(V)$ is an augmentation. 
With a little abuse of language, \emph{by a morphism $f\colon S^c(V)\to S^c(W)$ of symmetric coalgebras we shall mean a morphism of graded augmented coalgebras}, i.e., a morphism of graded coalgebras such that $f(1)=1$. 
In practice the assumption $f(1)=1$ is equivalent to the fact that $f$ is non trivial:  
it is an easy exercise to show that, given a morphism of graded coalgebras 
$f\colon S^c(V)\to S^c(W)$, then either $f=0$ or $f(1)=1$.

The following propositions are well known (see e.g. \cite{LadaMarkl,ManRendiconti}) and, in any case, easy to prove.

\begin{proposition} Every  morphism of symmetric  coalgebras 
$f\colon S^c(V)\to S^c(W)$ is uniquely determined by its corestriction, i.e., $f$ is uniquely determined by the components $f^1_j$, $j>0$. Moreover, 
for every $n>0$ and
$v_1,\ldots,v_n\in V$ we have 
\[ f(v_1\odot\cdots\odot v_n)-f_1^1(v_1)\odot\cdots\odot f_1^1(v_n)\in  
\bigoplus_{0<i<n}W^{\odot i}\;.\]
In particular $f$ is an isomorphism if and only if $f^1_1$ is an isomorphism.
\end{proposition}

For every graded coalgebra $C$ we shall denote by $\Coder^*(C,C)$ the graded Lie algebra of 
coderivations of $C$.

\begin{proposition}\label{prop.corestrictionderivation} 
The corestriction map gives an isomorphism of graded vector spaces
\[ \Coder^*(S^c(V),S^c(V))\to \Hom^*_{\K}(S^c(V),V)=
\prod_{k\ge 0}\Hom^*_{\K}(V^{\odot k},V)\]
whose inverse map 
\[ \Hom^*_{\K}(V^{\odot k},V)\ni \; q\mapsto \widehat{q}\; \in \Coder^*(S^c(V),S^c(V))\]
is described explicitly by the formulas
\[ \widehat{q}\,(v_1\odot\cdots\odot v_n)=
\sum_{\sigma\in S(k,n-k)}\epsilon(\sigma) 
q(v_{\sigma(1)}\odot\cdots\odot  v_{\sigma(k)})\odot 
v_{\sigma(k+1)}\odot\cdots\odot  v_{\sigma(n)}\;.\]
\end{proposition}

For $k=0$ the formula of Proposition~\ref{prop.corestrictionderivation} should be interpreted in the following sense: if $q\in \Hom^*_{\K}(V^{\odot 0},V)\cong \Hom^*_{\K}(\K,V)$, then 
\[ \widehat{q}(1)=q(1),\qquad \widehat{q}(v_1\odot\cdots\odot v_n)=q(1)\odot v_1\odot\cdots\odot v_n\;.\]
Therefore, for $q\in \Hom^*_{\K}(V^{\odot k},V)$ we have $\widehat{q}(V^{\odot n})\subseteq 
V^{\odot n-k+1}$.

The graded commutator  on $\Coder^*(S^c(V),S^c(V))$ induces, 
via the corestriction isomorphism, a bracket 
\[ [-,-]_{NR}\colon \Hom^*_{\K}(S^c(V),V)\times 
\Hom^*_{\K}(S^c(V),V)\to \Hom^*_{\K}(S^c(V),V)\;,\]
known as  Nijenhuis--Richardson bracket. In the notation of 
Proposition~\ref{prop.corestrictionderivation} we have 
\[ [-,-]_{NR}\colon \Hom^*_{\K}(V^{\odot n},V)\times 
\Hom^*_{\K}(V^{\odot m},V)\to \Hom^*_{\K}(V^{\odot n+m-1},V)\;,\]
\[ [f,g]_{NR}=f\,\widehat{g}-(-1)^{\bar{f}\;\bar{g}}g\,\widehat{f}\;.\]

\begin{definition} An $L_{\infty}[1]$-algebra is a graded vector space $V$ equipped with 
a coderivation of degree $+1$,  $Q\in \Coder^1(S^c(V),S^c(V))$ such that $Q(1)=0$ and $QQ=\frac{1}{2}[Q,Q]=0$. 
An $L_{\infty}$-morphism $f\colon (V,Q)\dashrightarrow (W,R)$ 
of $L_{\infty}[1]$-algebras is a morphism of symmetric  coalgebras
$f\colon S^c(V)\to S^c(W)$ such that $fQ=Rf$.
\end{definition}

Thus, see e.g. \cite{K}, there exists a canonical bijection between the set of $L_{\infty}$-algebra 
structures of a graded vector space $V$ and the he set of $L_{\infty}[1]$-algebra 
structures of a graded vector space $V[1]$. In particular every result about $L_{\infty}$-algebras 
holds, mutatis mutandis, also for $L_{\infty}[1]$-algebras.

Clearly we can define an $L_{\infty}[1]$-algebra also in terms of the 
Nijenhuis--Richardson bracket: more precisely 
an  $L_{\infty}[1]$-algebra is an $\infty$-uple
$(V,q_1,q_2,\ldots)$, where $q_n\in \Hom^1_{\K}(V^{\odot n},V)$ are such that for every $n>0$
\[\sum_{k=1}^{n-1}[q_k,q_{n-k}]_{NR}=0\;;\] 
the relation between the above two definitions is given by 
\[(V,Q)\mapsto (V,Q^1_1,Q^1_2,Q^1_3,\ldots),\qquad
(V,q_1,q_2,\ldots)\mapsto \left(V,\sum_{i>0} \widehat{q_i}\right)\;.\]

Notice that, if $(V,q_1,q_2,\ldots)$ is an $L_{\infty}[1]$-algebra we have $q_1q_1=0$ and then 
$(V,q_1)$ is a complex of vector spaces; we shall denote by $H^*(V)$ its cohomology, called the tangent cohomology of $V$. Since the equation $[q_1,q_2]_{NR}=0$ may be written as 
\[ q_1(q_2(x,y))+q_2(q_1(x),y)+(-1)^{\bar{x}}q_2(x,q_1(y))=0\]
we have that $q_2$ factors to a graded commutative (quadratic)
bracket  on tangent cohomology:
\[ q_2\colon H^*(V)\times H^*(V)\to H^*(V)\;.\]
It is well known and in any case easy to prove that, for every $L_{\infty}$-morphism 
$f\colon (V,q_1,q_2,\ldots)\dashrightarrow (W,r_1,r_2,\ldots)$ its linear component
\[ f^1_1\colon (V,q_1)\to (W,r_1)\]
is a morphism of complexes whose restriction to tangent cohomology
\[ f^1_1\colon H^*(V)\to H^*(W)\]
commutes with the quadratic brackets $q_2,r_2$.

\begin{example}[D\'ecalage]\label{ex.decalageofDGLA} 

The d\'ecalage functor,  from the category of differential graded Lie algebras to the category of $L_{\infty}[1]$-algebra is defined as 
\[ (L,d,[-,-])\mapsto (V,q_1,q_2,0,0,\ldots)\] 
where:
\begin{enumerate} 
\item $V$ is a graded vector space equipped with a linear map  $s\colon V\to L$, of degree $+1$, inducing an isomorphism  $s\colon V^i\xrightarrow{\;\simeq\;}L^{i+1}$ for every $i$;

\item the maps $q_1,q_2$ are defined  by the formulas:
\[ sq_1(v)=-d(sv),\qquad sq_2(u,v)=-(-1)^{\bar{u}}[su,sv],\qquad u,v\in V\;.\]
\end{enumerate}

Conversely,  every $L_{\infty}[1]$-algebra $(V,q_1,q_2,\ldots)$ such that $q_i=0$ for every $i\ge 3$ is the 
d\'ecalage  of a differential graded Lie algebra. 
\end{example}

\begin{definition} An $L_{\infty}$-morphism $f\colon (V,q_1,q_2,\ldots)\dashrightarrow (W,r_1,r_2,\ldots)$ is called a weak equivalence if induces an isomorphism between tangent cohomology groups
$f^1_1\colon H^*(V)\mapor{\simeq}H^*(W)$.
\end{definition}

\begin{definition} An $L_{\infty}[1]$-algebra $(V,q_1,q_2,\ldots)$ is said to be minimal if 
$q_1=0$.
\end{definition}

In other words, an $L_{\infty}[1]$-algebra $V$ is minimal if and only if $H^*(V)=V$.

\begin{theorem}[Minimal model theorem]\label{thm.minimalmodel} 
For every  
$L_{\infty}[1]$-algebra $(V,q_1,q_2,\ldots)$ there exist a minimal 
$L_{\infty}[1]$-algebra $(W,0,r_2,\ldots)$ and two weak equivalences 
\[ f\colon (V,q_1,q_2,\ldots)\dashrightarrow (W,0,r_2,\ldots),\qquad g\colon 
(W,0,r_2,\ldots)\dashrightarrow (V,q_1,q_2,\ldots)\]
such that $fg$ is the identity on $W$.
\end{theorem}

\begin{proof} See e.g. either Lemma 4.9 of \cite{K}, or 
Theorem 3.0.9 of \cite{kontsevichSoibelmanbook}.\end{proof}

The $L_{\infty}[1]$-algebra $(W,0,r_2,\ldots)$  as in the above theorem is unique up to 
isomorphisms and it  is called 
the \emph{minimal model} of $(V,q_1,q_2,\ldots)$.  
It is worth to mention that, as a consequence of Theorem~\ref{thm.minimalmodel}, if 
$f\colon (V,q_1,q_2,\ldots)\dashrightarrow (W,r_1,r_2,\ldots)$ is a weak equivalence, then there exists a weak equivalence $g\colon (W,r_1,r_2,\ldots)\dashrightarrow (V,q_1,q_2,\ldots)$ such that 
$g^1_1\colon H^*(W)\to H^*(V)$ is the inverse of $f^1_1\colon H^*(V)\to H^*(W)$.

\begin{definition} An $L_{\infty}[1]$-algebra $(V,q_1,q_2,\ldots)$ is said to be formal  
if it is weak equivalent to a (pure quadratic) $L_{\infty}[1]$-algebra $(W,0,r_2,0,0,\ldots)$.
It is called homotopy abelian if it is weak equivalent to a trivial 
$L_{\infty}[1]$-algebra $(W,0,0,0,0,\ldots)$
\end{definition}

Therefore, an $L_{\infty}[1]$-algebra $(V,q_1,q_2,\ldots)$ is  formal
if and only if its minimal model is isomorphic to  $(H^*(V),0,r_2,0,0,\ldots)$, where 
$r_2$ is the restriction of the quadratic component $q_2$ to the tangent cohomology $H^*(V)$.

It is well known, see e.g. \cite{hinichDGC}, 
that two differential graded Lie algebras are quasi-isomorphic 
if and only if they have  isomorphic $L_{\infty}[1]$ minimal models. In particular a differential graded Lie algebra is formal (resp.: homotopy abelian) if and only if the associated $L_{\infty}[1]$-algebra is formal (resp.: homotopy abelian).

\bigskip
\section{Homotopy invariance of Chevalley--Eilenberg spectral sequence}
\label{sec.homotopyinvariance}

Let's recall, following \cite{Gode}, 
the detailed construction of the  spectral sequence associated to a differential filtered complex.

Let $M$ be an abelian group, equipped with a homomorphism $d\colon M\to M$ such that $d^2=0$ 
and a decreasing filtration $F^pM$, $p\in \Z$, such that $d(F^pM)\subset F^pM$. 
The associated spectral sequence $(E^{p}_r,d_r)$, $r\ge 0$, is defined as
\[ Z^{p}_r=\{x\in F^pM\mid d x\in F^{p+r}M\},\qquad 
E^{p}_r=\frac{Z^{p}_r}{Z^{p+1}_{r-1}+dZ^{p-r+1}_{r-1}}\;,\]
and the maps 
\[d_r\colon E^{p}_r\to E^{p+r}_r\] 
are induced by $d$ in the obvious way. 
We have
$d_r^2=0$ and there exist natural isomorphisms 
\[ E^{p}_{r+1}\cong \frac{\ker(d_r\colon E^{p}_r\to E^{p+r}_r)}{d_r(E^{p-r}_r)}.\]

If $M=\oplus M^n$ is graded, $d(M^n)\subset M^{n+1}$ and every $F^pM=\oplus_n F^pM^n$ 
is a graded subgroup, then 
every group $E^p_r$ inherits a natural graduation, namely: 
\[ Z^{p,q}_r=\{x\in F^pM^{p+q}\mid d x\in F^{p+r}M^{p+q+1}\},\qquad 
E^{p,q}_r=\frac{Z^{p,q}_r}{Z^{p+1,q-1}_{r-1}+dZ^{p-r+1,q+r-2}_{r-1}}\;,\]
~\\
\[E^p_r=\bigoplus_{q} E^{p,q}_r,\qquad d_r\colon E^{p,q}_r\to E^{p+r,q-r+1}_r,\qquad
E^{p,q}_{r+1}\cong \frac{\ker(d_r\colon E^{p,q}_r\to E^{p+r,q-r+1}_r)}{d_r(E^{p-r,q+r-1}_r)}\;.\]

It is convenient to introduce  a refinement of the usual notion of degeneration of a spectral sequence
\cite{DGMS}.

\begin{definition}\label{def.collapsing1} 
We shall say that a cohomology spectral sequence $(E^{p,q}_r,d_r)$ degenerates at $E^{a,b}_k$ if 
the map $d_r\colon E^{a,b}_r\to E^{a+r,b-r+1}_r$ vanishes for every $r\ge k$.
A spectral sequence $(E^{p,q}_r,d_r)$ degenerates at $E_k$ if $d_r=0$ for every $r\ge k$.
\end{definition}

\bigskip

For every morphism of graded coalgebras $f\colon C\to D$ we shall denote by 
$\Coder^*(C,D;f)$ the graded vector space of coderivations $\alpha\colon C\to D$, with
the structure of $D$-comodule on $C$ induced by the morphism $f$. 
When $f$ is the identity we shall simply denote $\Coder^*(C,C)=\Coder^*(C,C;\Id_C)$.

\begin{definition} The  Chevalley--Eilenberg complex of an $L_{\infty}$-morphism 
$f\colon (V,Q)\dashrightarrow (W,R)$ of $L_{\infty}[1]$-algebras is the filtered differential 
complex 
\[ \begin{split}
CE(V,W;f)&=\Coder^*(S^c(V),S^c(W);f)\\
&=\{\alpha\in \Hom^*_{\K}(S^c(V),S^c(W))\mid 
\Delta\alpha=(\alpha\otimes f+f\otimes\alpha)\Delta\},\end{split}
\]
where:
\begin{enumerate}

\item the filtration is defined as 
\[ F^p{CE}(V,W;f)=\{\alpha\in \Coder^*({S^c}(V),{S^c}(W);f)\mid \alpha(V^{\odot i})=0,
\quad \forall\; i<p\};\]

\item  the differential $d$ is defined by the formula
\[ d\alpha=R\alpha-(-1)^{\bar{\alpha}}\alpha Q\;.\]

\end{enumerate}
\end{definition}

As in 
Proposition~\ref{prop.corestrictionderivation}, 
the corestriction map gives an isomorphism of graded vector spaces
\[ \Coder^*(S^c(V),S^c(W);f)\to \Hom^*_{\K}(S^c(V),W)=
\prod_{k\ge 0}\Hom^*_{\K}(V^{\odot k},W),\]
although the inverse map $\alpha\mapsto \widehat{\alpha}$   
is now described explicitly by a more complicated formula, cf. \cite{ManRendiconti}.
However for  our applications  we only need the description of $\widehat{\alpha}$ 
in some particular and easy cases, namely:  
\begin{enumerate}

\item for $w\in W$ the associated coderivation 
$\widehat{w}\in \Coder^*({S^c}(V),{S^c}(W);f)$ satisfies the equalities
\[ \widehat{w}(1)=w,\qquad \widehat{w}(v)=w\otimes f^1_1(v),\qquad v\in V\;.\]

\item for $\alpha\in \Hom^*_{\K}(V,W)$ the corresponding coderivation satisfies $\widehat{\alpha}(1)=0$, 
$\widehat{\alpha}(v)=\alpha(v)$ and
\[  
\widehat{\alpha}(v_1\odot v_2)=\alpha(v_1)\odot f_1^1(v_2)+(-1)^{\bar{v_1}\;\bar{v_2}}
\alpha(v_2)\odot f_1^1(v_1),\qquad v_1,v_2\in V\;.\]
\end{enumerate}

The cohomology spectral sequence of the  
filtered differential complex ${CE}(V,W;f)$ defined above will be 
denoted by $(E(V,W;f)^{p,q}_r,d_r)$.

\begin{lemma}\label{lem.criteriodicollassamentoparziale} 
Let $(E(V,W;f)_r,d_r)$ be the Chevalley--Eilenberg spectral sequence of 
an $L_{\infty}$-morphism  
$f\colon (V,0,q_2,q_3,\ldots)\dashrightarrow (W,0,r_2,r_3,\ldots)$  of minimal $L_{\infty}[1]$-algebras. Assume that for some integer $k\ge 3$ we have 
$q_3=\cdots=q_k=0$ and 
$r_3=\cdots=r_k=0$. Then $d_r=0$ for every $2\le r<k$ and therefore
\[ E(V,W;f)_2=E(V,W;f)_3=\cdots=E(V,W;f)_{k}\;.\]

\end{lemma}

\begin{proof} By induction it is sufficient to prove that $d_{k-1}=0$. 
We shall write $Q=Q_2+Q'$ for the codifferential of ${S^c}(V)$, where 
\[ Q_2=\widehat{q_2},\quad Q'=\sum_{i>k}\widehat{q_i},\qquad
Q_2(V^{\odot i})\subset V^{\odot i-1},\qquad 
Q'(V^{\odot i})\subset \oplus_{j\le i-k}V^{\odot j}\;.\]
Similarly we write $R=R_2+R'$, where
\[ R_2=\widehat{r_2},\quad R'=\sum_{i>k}\widehat{r_i},\qquad
R_2(W^{\odot i})\subset W^{\odot i-1},\qquad 
R'(W^{\odot i})\subset \oplus_{j\le i-k}W^{\odot j}\;.\]

An element $v\in E(V,W;f)^p_{k-1}$ 
is represented by a linear map $\alpha\in \Hom^*_{\K}(V^{\odot p},W)$ such that   
$R\widehat{\alpha}-(-1)^{\bar{\alpha}}\widehat{\alpha}Q\in  F^{p+k-1}\Coder^*({S^c}(V),{S^c}(W);f)$.
Since $k\ge 3$ we must have $R_2\widehat{\alpha}-(-1)^{\bar{\alpha}}\widehat{\alpha}Q_2=0$,
\[R\widehat{\alpha}-(-1)^{\bar{\alpha}}\widehat{\alpha}Q=R'\widehat{\alpha}-(-1)^{\bar{\alpha}}\widehat{\alpha}Q'\in
F^{p+k}\Coder^*({S^c}(V),{S^c}(W);f)\]
and this implies that $d_{k-1}(v)=0$.
\end{proof}

\begin{lemma}\label{lem.composizionemorfismi} 
Let $f\colon (V,Q)\dashrightarrow (W,R)$ and 
$g\colon (W,R)\dashrightarrow (U,S)$ be $L_{\infty}$-morphisms of $L_{\infty}[1]$-algebras.
Then the composition maps give two morphisms of filtered differential complexes
\[ g_*\colon \Coder^*({S^c}(V),{S^c}(W);f)\to \Coder^*({S^c}(V),{S^c}(U);gf),\]
\[ f^*\colon \Coder^*({S^c}(W),{S^c}(U);g)\to \Coder^*({S^c}(V),{S^c}(U);gf).\]
\end{lemma}

\begin{proof}
Since $f(V^{\odot n})\subset \oplus_{i\le n}W^{\odot i}$ it is obvious that $f^*$ and $g_*$ preserve 
the filtrations.
Now
\[ g(d\alpha)=gR\alpha-(-1)^{\bar{\alpha}}\alpha Q=Sg\alpha-(-1)^{\bar{\alpha}}g\alpha Q=d(g\alpha)\]
\[ (d\alpha)f=S\alpha f-(-1)^{\bar{\alpha}}\alpha Rf=S\alpha f-
(-1)^{\bar{\alpha}}\alpha fQ=d(\alpha f).\]
\end{proof}

\begin{proposition}\label{prop.confrontoss} 
In the situation of  Lemma~\ref{lem.composizionemorfismi} the composition maps induce two morphisms of spectral sequences
\[ \xymatrix{E(W,U;g)^{p,q}_r\ar[r]^{f^*}&E(V,U;gf)^{p,q}_r&E(V,W;f)^{p,q}_r\ar[l]_{g_*}}\;:\]

\begin{enumerate} 

\item if $f$ is a weak equivalence, then $f^*\colon E(W,U;g)^{p,q}_r\to E(V,U;gf)^{p,q}_r$ is an isomorphism for every $r\ge 1$;

\item if $g$ is a weak equivalence, then $g_*\colon E(V,W;f)^{p,q}_r\to E(V,U;gf)^{p,q}_r$ 
is an isomorphism for every $r\ge 1$.
\end{enumerate}
\end{proposition}

\begin{proof}  For every integer $p$, the  corestriction map gives an isomorphism
\[ F^p{CE}(V,W;f)\simeq \prod_{n\ge p}\Hom^*_{\K}(V^{\odot n},W)\]
and therefore 
\[ E(V,W;f)^{p,q}_0=\Hom^{p+q}_{\K}(V^{\odot p},W)\;.\]

Given an element $\alpha\in\Hom^*_{\K}(V^{\odot p},W)$ we have $\alpha=\widehat{\alpha}_{|V^{\odot p}}$, 
$\widehat{\alpha}(V^{\odot i})=0$ for every $i<p$ and then 
\[ \widehat{\alpha}Q(v_1\odot\cdots\odot v_n)=
\alpha\left(\sum_{i=1}^n(-1)^{\bar{v_1}+\cdots+\bar{v_{i-1}}}v_1\odot\cdots\odot Q^1_1(v_i)
\odot\cdots\odot v_n\right)\;\] 
\[ R\widehat{\alpha}(v_1\odot\cdots\odot v_n)=R^1_1\alpha(v_1\odot\cdots\odot v_n)\;,\]
\[ d_0\alpha=R^1_1\alpha-(-1)^{\bar{\alpha}}\alpha
\left(\sum_{i=1}^n \Id^{\odot i-1}\odot Q^1_1\odot \Id^{\odot n-i}\right)\;.\]
In other terms $d_0$ is the standard  differential in $\Hom^*_{\K}(V^{\odot n},W)$; by 
K\"{u}nneth formula (cf. \cite[pag. 280]{Qui})
\[ E(V,W;f)^{p,q}_1=\Hom^{p+q}_{\K}(H^*(V)^{\odot p},H^*(W)).\]
The conclusion of the proof is now clear.
\end{proof}

\begin{definition}\label{def.eulerderivationLinfinity} 
The \emph{Euler derivation}  of an $L_{\infty}$-morphism $f\colon V\dashrightarrow W$ is the element
\[ e_f \in E(V,W;f)^{1,-1}_1=\Hom^0_{\K}(H^*(V),H^*(W))\]
defined as 
\[ e_f(v)=(\bar{v}+1)f^1_1(v),\qquad v\in H^*(V)\;.\]
\end{definition}

We are now ready to prove the main result of this section.

\begin{theorem}\label{thm.invarianzaalminimale}
Let $W$ be the minimal model of an $L_{\infty}[1]$-algebra $V$. 
Then there exists a morphism of spectral sequences 
\[ E(V,V)^{p,q}_r\to E(W,W)^{p,q}_r\]
which is an isomorphism for every $r\ge 1$ and preserves the Euler derivations.
\end{theorem}

\begin{proof} By minimal model theorem there exist two weak equivalences 
\[ g\colon V\dashrightarrow W,\qquad f\colon 
W\dashrightarrow V\]
such that $gf$ is the identity on $W$. It now sufficient to consider the morphisms
\[ E(V,V)^{p,q}_r\mapor{f^*}E(W,V;f)^{p,q}_r\mapor{g_*}E(W,W;gf)^{p,q}_r=E(W,W)^{p,q}_r\]
and apply Proposition~\ref{prop.confrontoss}.
\end{proof}

\begin{lemma} Let $e_f\in E(V,W;f)^{1,-1}_1$ be the Euler derivation of an $L_{\infty}$-morphism. 
Then $d_1(e_f)=0$.
\end{lemma} 

\begin{proof} By Proposition~\ref{prop.confrontoss} it is not restrictive to assume both $V$ and $W$ minimal $L_{\infty}[1]$-algebras, say $V=(V,0,q_2,\ldots)$ and $W=(W,0,r_2,\ldots)$.
Let's give an explicit description of the two differentials
\[\xymatrix{ E(V,W;f)^{0,0}_1\ar[r]^{d_1}\ar@{=}[d]&E(V,W;f)^{1,-1}_1\ar[r]^{d_1}\ar@{=}[d]&
E(V,W;f)^{2,-2}_1\ar@{=}[d]\\
W^0&\Hom^0_{\K}(V,W)&\Hom^0_{\K}(V^{\wedge 2},W)}\]
Given $w\in W^0$ the associated coderivation $\widehat{w}\in \Coder^*({S^c}(V),{S^c}(W);f)$ satisfies the equalities
\[ \widehat{w}(1)=w,\qquad \widehat{w}(v)=w\otimes f^1_1(v),\qquad v\in V\] 
and then, 
\[ (d_1w)(v)=r_2(w,f^1_1(v))\in W\;.\]
Similarly for $\alpha\in \Hom^0_{\K}(V,W)$ the corresponding coderivation satisfies 
$\widehat{\alpha}(v)=\alpha(v)$, 
\[  
\widehat{\alpha}(v_1\odot v_2)=\alpha(v_1)\odot f_1^1(v_2)+(-1)^{\bar{v_1}\;\bar{v_2}}
\alpha(v_2)\odot f_1^1(v_1)=\alpha(v_1)\odot f_1^1(v_2)+
 f_1^1(v_1)\odot \alpha(v_2)\;,\]
and therefore 
\[ (d_1\alpha)(v_1\odot v_2)=r_2(\alpha(v_1),f_1^1(v_2))+
r_2(f_1^1(v_1),\alpha(v_2))-\alpha(q_2(v_1, v_2))\;.\]
In particular 
\[\begin{split} 
(d_1&e_f)(v_1\odot v_2)=\\
&=r_2((\bar{v_1}+1)f^1_1(v_1), f_1^1(v_2))+
r_2(f_1^1(v_1),(\bar{v_2}+1)f^1_1(v_2))-(\bar{v_1}+\bar{v_2}+2)f^1_1(q_2(v_1, v_2))\\
&=0\;.\qquad
\end{split}\]
\end{proof}

\begin{definition}\label{def.eulerclassLinfinity} 
The \emph{Euler class}  of an $L_{\infty}$-morphism $f\colon V\dashrightarrow W$ is the element
\[ e_f \in E(V,W;f)^{1,-1}_2\]
defined as the class of the Euler derivation modulus $d_1(W^0)$. 
The \emph{Euler class}  of an $L_{\infty}$-algebra is 
the Euler class of the identity.
\end{definition}

It is plain from the above results that the Euler class of an $L_{\infty}[1]$-algebra 
is invariant under weak equivalence.

\bigskip

\section{Formality criteria for $L_{\infty}[1]$-algebras}
\label{sec.criteriaLinfinity}

\begin{lemma}\label{lem.formalityimpliesdegeneration} 
Let $k$ be a positive integer and let $(V,q_1,q_2,\ldots)$ be an  
$L_{\infty}[1]$-algebra such that $q_i=0$ for every $i\not=k$. Then:
\begin{enumerate}
\item  the spectral 
sequence $E(V,V)^{p,q}_r$ degenerates at $E_k$;

\item if the spectral 
sequence $E(V,V)^{p,q}_r$ degenerates at $E^{1,-1}_{k-1}$ then also $q_k=0$.
\end{enumerate}

\end{lemma}

\begin{proof} Via the identification ${CE}(V,V)\cong\Hom^*_{\K}
({S^c}(V),V)$ the differential of the complex becomes  
$d=[q_k,-]_{NR}$ and then $d(\Hom^*_{\K}
(V^{\odot n},V))\subset \Hom^*_{\K}
(V^{\odot n+k-1},V)$. 
Therefore if $\alpha=\sum \alpha_i$, with $\alpha_i\in \Hom^*_{\K}(V^{\odot i},V)$ and
$d\alpha\in F^{n+k}{CE}(V,V)$ then $d\alpha_i=0$ for $i\le n$; 
therefore, for every $n\in \Z$ and every $r\ge k$ we have
\[ dZ^{n}_r\subset   
d(F^{n+1}{CE}(V,V))\cap F^{n+r}{CE}(V,V)=dZ^{n+1}_{r-1}\]
and this implies $d_r=0$.

As regards the second item, we assume $k>1$, being the case $k=1$ completely trivial. 
Considering the identity map $\Id_V\in \Hom^0_{\K}(V,V)$ as an element of 
$F^1{CE}(V,V)$ 
we have:
\[q_k=\frac{1}{k-1}[q_k,\Id_V]_{NR}\in F^k{CE}(V,V),\qquad \Id_V\in Z^{1,-1}_{k-1}\;.\]

If the spectral sequence degenerates at $E^{1,-1}_{k-1}$, then the class of $q_k$ is trivial 
in $E_{k-1}^{k,1-k}$; in particular 
\[q_k\in d(F^2{CE}(V,V))+F^{k+1}{CE}(V,V)\]
and this implies $q_k=0$.
\end{proof}

\begin{lemma}\label{lem.lemmachiave} 
Let $(V,0,q_2,0,\ldots,0,q_i,\ldots)$ be a minimal $L_{\infty}[1]$-algebra such that 
$q_j=0$ for every $2<j<i$ and some $2<i$. Hence, by Lemma~\ref{lem.criteriodicollassamentoparziale} we have  $E(V,V)_2^{p,q}=E(V,V)^{p,q}_{i-1}$.
Denoting by $e\in E(V,V)^{1,-1}_2$ its Euler class we have:

\begin{enumerate}

\item $[q_2,q_i]_{NR}=0$;

\item $d_r(e)=0\in E(V,V)^{r+1,-r}_r$ for every $1\le r<i-1$;

\item if $d_{i-1}(e)=0\in E(V,V)^{i,1-i}_{i-1}=E(V,V)^{i,1-i}_{2}$, then 
there exists $\alpha\in \Hom^0_{\K}(V^{\odot i-1},V)$ such that 
$q_i=[q_2,\alpha]_{NR}$.  
\end{enumerate}

\end{lemma}

\begin{proof}
The first part is clear since $\sum_j [q_j, q_{i+2-j}]_{NR}=0$. 
The Euler class  $e$ is induced by 
the linear map
\[ e\in \Hom^0_{\K}(V,V),\qquad e(v)=(\bar{v}+1)v\;,\]
and for every  $\beta\in \Hom^h_{\K}(V^{\odot j},V)$ we have 
\[ [\beta,e]_{NR}=(j-h-1)\beta\;.\]

Therefore, setting $q=\sum q_j$, we have  $[q_2,e]_{NR}=0$, 
\[ [q,e]_{NR}=(i-2)q_i+(i-1)q_{i+1}+\cdots\quad\in F^{i}{CE}(V,V)\cap d(F^1{CE}(V,V)).\]
In particular $[q,e]_{NR}\in Z^{j+1,-j}_{j}$ for every $j<i$ and this implies 
that  $d_r(e)=0\in E(V,V)^{r+1,-r}_r$ for $1\le r<i-1$.
If $d_{i-1}(e)=0\in E^{i,1-i}_{i-1}$, then 
\[ (i-2)q_i+(i-1)q_{i+1}+\cdots\quad\in Z^{i+1,-i}_{i-2}+dZ^{2,-1}_{i-2}\]
and then there exists  a sequence 
$\alpha_j\in \Hom^0_{\K}(V^{\odot j},V)$, $j\ge 2$, such that 
\[ [q,\sum\alpha_j]_{NR}-(i-2)q_i\in F^{i+1}{CE}(V,V)\;\]
this implies that $[q_2,\alpha_j]_{NR}=0$ for $j<i-1$ and 
$[q_2,\alpha_{i-1}]_{NR}=(i-2)q_i$. 
In particular, $\alpha=\dfrac{\alpha_{i-1}}{i-2}$ is the required element.
\end{proof}

\begin{theorem}\label{thm.formalityminilamlLinfinity1} 
For a minimal $L_{\infty}[1]$-algebra 
$(V,0,q_2,q_3,\ldots)$ with Euler class $e\in E(V,V)_2^{1,-1}$, the following conditions are equivalent:

\begin{enumerate} 

\item\label{item1.thm.formalityminilamlLinfinity1} 
there exists an $L_{\infty}$-isomorphism $f\colon 
(V,0,q_2,0,0,\ldots)\dashrightarrow (V,0,q_2,q_3,\ldots)$;

\item\label{item2.thm.formalityminilamlLinfinity1}  the spectral sequence $E(V,V)^{p,q}_r$ degenerates at $E_2$;

\item\label{item3.thm.formalityminilamlLinfinity1}  $d_r(e)=0\in E(V,V)^{r+1,-r}_r$ for every $r\ge 2$.

\end{enumerate}
\end{theorem}

\begin{proof} By Lemma~\ref{lem.formalityimpliesdegeneration} and the homotopy invariance
of the Chevalley--Eilenberg spectral sequence, 
we only need to prove $(\ref{item3.thm.formalityminilamlLinfinity1}\Rightarrow 
\ref{item1.thm.formalityminilamlLinfinity1})$. If $q_i=0$ for every $i>2$ there is nothing to prove, 
otherwise let $i\ge 3$ be the smallest integer such that $q_i\not=0$; 
by Lemma~\ref{lem.lemmachiave} 
there exists an operator $\alpha\in \Hom^0_{\K}(V^{\odot i-1},V)$ 
such that $[q_2,\alpha]_{NR}=q_i$.
Denoting by $\widehat{\alpha}\in \Coder^0({S^c}(V),{S^c}(V))$ the corresponding 
pronilpotent coderivation, by $Q=\sum_j\widehat{q_j}$ and by  
\[ R=e^{-\widehat{\alpha}}Qe^{\widehat{\alpha}}=
e^{[\widehat{\alpha},-]}(Q)=Q+[\widehat{\alpha},Q]+\frac{1}{2}[\widehat{\alpha},[\widehat{\alpha},Q]]+\cdots
\]
we have that $e^{\widehat{\alpha}}\colon (V,R)\dashrightarrow (V,Q)$ is an $L_{\infty}$-morphism.
Denoting by $r$ and $q=\sum q_i$ the corestrictions of $R$ and $Q$, respectively, we have:
\[ r=e^{[\alpha,-]_{NR}}(q)=q+[\alpha,q]_{NR}+\cdots\equiv q_2+q_i-[q_2,\alpha]_{NR}\equiv q_2\quad 
\left(\mod{\prod_{j>i}\Hom^1_{\K}(V^{\odot j},V)}\right).\]
Therefore $r=q_2+r_{i+1}+r_{i+2}+\cdots$ and then we have an $L_{\infty}$-isomorphism
\[ e^{\widehat{\alpha}}\colon (V,q_2,0,\ldots,0,r_{i+1},\ldots)\dashrightarrow 
(V,q_2,0,\ldots,0,q_i,q_{i+1},\ldots).\]
Since $e^{\widehat{\alpha}}$ is the identity on $V^{\odot j}$ for every $j\le i-2$, we can repeat 
the procedure infinitely many times and take $f$ as the infinite composition product 
of the above exponentials.  

\end{proof}

\begin{corollary}\label{cor.abelianitycriterionminimal}  
For a minimal $L_{\infty}[1]$-algebra 
$(V,0,q_2,q_3,\ldots)$ the following conditions are equivalent:

\begin{enumerate} 

\item $q_i=0$ for every $i$;

\item the spectral sequence $E(V,V)^{p,q}_r$ degenerates at $E_1$;

\item the spectral sequence $E(V,V)^{p,q}_r$ degenerates at $E^{1,-1}_1$.
\end{enumerate}
\end{corollary}

\begin{proof} Immediate consequence of Theorem~\ref{thm.formalityminilamlLinfinity1} and 
Lemma~\ref{lem.formalityimpliesdegeneration}.
\end{proof}

\begin{remark} In \cite{kaledin} it is  proved (in the framework of $A_{\infty}$-algebras) 
that 
property \eqref{item1.thm.formalityminilamlLinfinity1} of Theorem~\ref{thm.formalityminilamlLinfinity1}  is equivalent to the vanishing of a certain cohomology class, called ``Kaledin class'' in \cite{lunts}.
For a minimal $L_{\infty}[1]$-algebras $(V,0,q_2,\ldots)$,  the Kaledin class may  be defined in the following way:
let $t$ be a central formal indeterminate of degree 0; by 
extending the Nijenhuis--Richardson bracket to ${CE}(V,V)[[t]]$ in the obvious way, by homogeneity we have   
\[ [q(t),q(t)]_{NR}=0,\qquad \text{where}\quad q(t)=q_2+tq_3+t^2q_4+\cdots\]
and then $d(t)=[q(t),-]_{NR}$ is a differential on the 
$\K[[t]]$-module ${CE}(V,V)[[t]]$. 

Taking the formal derivative on the variable $t$ we have 
\[\de_t q(t)=q_3+2tq_4+\cdots,\qquad 
[q(t),\de_t q(t)]_{NR}=0\] 
and the cohomology class
\[[\de_t q(t)]\in H^1({CE}(V,V)[[t]], d(t))\]
is called the Kaledin class of the minimal $L_{\infty}[1]$-algebra $(V,0,q_2,\ldots)$;
it is supported at $t=0$, since  
\[ t\,\de_t q(t)=tq_3+2t^2q_4+\cdots=[q(t),e]_{NR}\]
where $e$ is Euler class.
In particular, the Kaledin class vanishes whenever the multiplication map 
$H^1({CE}(V,V)[[t]])\xrightarrow{\;t\;}H^1({CE}(V,V)[[t]])$
is injective. Notice that
there exists a short exact sequence  
of complexes
\[ 0\to  ({CE}(V,V)[[t]],d(t))\xrightarrow{\;t\;}({CE}(V,V)[[t]],d(t))\xrightarrow{t\mapsto 0}({CE}(V,V),[q_2,-]_{NR})\to 0\]
and it is an easy exercise on spectral sequences to prove that if 
$E(V,V)^{p,q}_r$ degenerates at $E^{a,-a}_2$ for every $a\ge 1$, then the map
\[H^0({CE}(V,V)[[t]],d(t))\xrightarrow{t\mapsto 0}H^0({CE}(V,V),[q_2,-]_{NR})\]
is surjective.

\end{remark}

The following two corollaries follow immediately from the above results together with Theorem~\ref{thm.invarianzaalminimale}.

\begin{corollary}\label{cor.formalityLinfinity} 
For an  $L_{\infty}[1]$-algebra 
$V$ with Euler class $e\in E(V,V)^{1,-1}_2$, the following conditions are equivalent:

\begin{enumerate} 

\item $V$ is formal;

\item the spectral sequence $E(V,V)^{p,q}_r$ degenerates at $E_2$;

\item $d_r(e)=0\in E(V,V)^{r+1,-r}_r$ for every $r\ge 2$.

\end{enumerate}
\end{corollary}

\begin{corollary}\label{cor.abelianitycriterion} 
For an  $L_{\infty}[1]$-algebra 
$V$ the following conditions are equivalent:

\begin{enumerate} 

\item $V$ is homotopy abelian;

\item the spectral sequence $E(V,V)^{p,q}_r$ degenerates at $E_1$;

\item the spectral sequence $E(V,V)^{p,q}_r$ degenerates at $E^{1,-1}_1$.
\end{enumerate}
\end{corollary}

The equivalence 
$(1\Leftrightarrow 2)$ of Corollary~\ref{cor.abelianitycriterion}, together some nice applications,  
has been recently proved by R. Bandiera \cite{bandierakapra} in a different way.

A well known result, which is very useful in deformation theory (see e.g. \cite{algebraicBTT,semireg2011,IaconoDP,KKP}) is that 
if $f\colon V\dashrightarrow W$ is an $L_{\infty}$-morphism, $W$ is homotopy abelian and 
$f^1_1\colon H^*(V)\to H^*(W)$ is injective, then also $V$ is homotopy abelian.
If $W$ is formal, then the injectivity in tangent cohomology is not sufficient to ensure the 
formality of $V$.  However, we have the following result, proved as a consequence of  
Theorem~\ref{thm.formalityminilamlLinfinity1}.

\begin{theorem}[Formality transfer]\label{thm.formalitytransferinfinity} 
Let $f\colon V\dashrightarrow W$ be an $L_{\infty}$-morphism of $L_{\infty}[1]$-algebras such that:
\begin{enumerate}

\item $W$ is formal;

\item the map $f_*\colon E(V,V)^{p,1-p}_2\to E(V,W;f)^{p,1-p}_2$ is injective for every 
$p\ge 3$.
\end{enumerate}
Then also $V$ is formal.
\end{theorem}

\begin{proof} It is not restrictive to assume  $V$  minimal and $W$ purely quadratic, say 
\[f\colon (V,0,q_2,q_3,\ldots)\dashrightarrow (W,0,r_2,0,0,\ldots)\;.\]
If $q_i\not=0$ for some $i>2$, let $k\ge 3$ be the smallest integer such that 
$q_k\not=0$; thus $q_i=0$ for every $2<i<k$.
According to Lemma~\ref{lem.criteriodicollassamentoparziale} we have 
\[ E(V,V)^{p,1-p}_2=E(V,V)^{p,1-p}_{k-1},\qquad E(V,W;f)^{p,1-p}_2=E(V,W;f)^{p,1-p}_{k-1}\]
and then also 
$f_*\colon E(V,V)^{p,1-p}_{k-1}\to E(V,W;f)^{p,1-p}_{k-1}$ is injective for every 
$p\ge 3$.
We have two morphisms of spectral sequences 
\[ \xymatrix{E(V,V)^{p,q}_r\ar[r]^{f_*\;}&E(V,W;f)^{p,q}_r&E(W,W)^{p,q}_r\ar[l]_{\;f^*}}\]
Denoting by $e_V,e_W$ and $e_f$ the Euler classes of $V,W$ and $f$ respectively we have 
$f_*(e_V)=e_f=f^*(e_W)$ and then for every $2\le i<k$
we have 
\[ f_*(d_ie_V)=d_i(f_*e_V)=d_i(f^*e_W)=f^*(d_ie_W)=0\in E(V,W;f)^{i+1,-i}_{i}\]
and then 
$d_ie_V=0\in E(V,V)^{i+1,-i}_{i}$ for every $2\le i<k$.
The same argument used in the proof of Theorem~\ref{thm.formalityminilamlLinfinity1} implies that, up to composition with an $L_{\infty}$-isomorphism of $V$ which is the identity 
on $V^{\odot i}$, $i<k-1$, we can assume $q_k=0$. Repeating this step, possibly infinitely many times 
for a sequence of increasing values of $k$, we prove the formality of $V$.   
\end{proof}

\bigskip
\section{Formality criteria for differential graded Lie algebras}
\label{sec.formalityDGLA}

By using the d\'ecalage isomorphism we can rewrite every formality result for $L_{\infty}[1]$-algebras 
in the framework of differential graded Lie algebras. 
If $V,L$ are graded vector spaces, then every bijective linear map 
$s\colon V\to L$ of degree $+1$ extends naturally  to a sequence of  linear isomorphisms
\[  s\colon \Hom_{\K}^n(V^{\odot k},V)\xrightarrow{\;\simeq\;} \Hom_{\K}^{n-k+1}(L^{\wedge k},L),\qquad n,k\in\Z,\quad k\ge 0\;, \]
defined by the formulas 
\begin{equation}\label{equ.formuladecalage} 
(s\phi)(sv_1,\ldots,sv_k)=(-1)^{\sum_i (k-i)\bar{v_i}}s\,\phi(v_1,\ldots,v_k),\qquad 
v_1,\ldots,v_n\in V\;.
\end{equation}

Assume now that $L$ is a differential graded Lie algebra and $(V,q_1,q_2,0,\ldots)$ its d\'ecalage, defined  as in Example~\ref{ex.decalageofDGLA}. A 
straightforward  computation shows that  the  bijective linear map of degree 
$+1$
\[ s\colon {CE}(V,V)\to  CE(L,L),\]
defined in \eqref{equ.formuladecalage} commutes, in the graded sense, with the differentials:
\[ s\bar{\delta}+[q_1,s(-)]_{NR}= 
s\delta+[q_2,s(-)]_{NR}=0\;.\]
In particular, $s$ gives 
a bijective morphism  of  spectral sequences of degree $+1$.
\[ s\colon E(V,V)^{p,q}_r\to E(L,L)^{p,q+1}_r\;\]
preserving the Euler classes.

After this, it is now clear that  Lemma~\ref{lem.degenerazioneparziale} follows from 
Lemma~\ref{lem.criteriodicollassamentoparziale}, 
Theorem~\ref{thm.formalitycriteria} is a direct 
consequence of Corollary~\ref{cor.formalityLinfinity}, while Theorem~\ref{thm.formalitytransfer} 
follows from  Theorem~\ref{thm.formalitytransferinfinity}.

\bigskip

One of the possible limitations in the application of Theorem~\ref{thm.formalitycriteria} 
is that in general, the filtration 
$F^pCE(L,L)$ is not bounded and  the spectral sequence $E(L,L)^{p,q}_r$ is not regular; 
thus it may be useful to restate our  results in terms of the spectral sequences of the 
quotient complexes $CE(L,L)/F^pCE(L,L)$.

\begin{lemma}\label{lem.ssdiquozienti} 
Let $F^pM$, $p\in\Z$, be a decreasing filtration of a differential graded abelian group $M$. 
Denote by $E^{p,q}_r$ the associated spectral sequence and, for every integer $b$, by
$E(l)^{p,q}_r$ the spectral sequence of the quotient filtered complex $M/F^lM$. 
Denoting by $\pi\colon M\to M/F^lM$ the projection, we have:

\begin{enumerate}

\item\label{it1.lem.ssdiquozienti} if $p<l$, then $\pi\colon E^{p}_r\to E(l)^{p}_r$ is injective;

\item\label{it2.lem.ssdiquozienti} if $p+r\le l$, then $\pi\colon E^{p}_r\to E(l)^{p}_r$ is surjective;

\item\label{it3.lem.ssdiquozienti} for a fixed pair $a,b$ of integers, the spectral sequences
$E^{p,q}_r$ degenerates at $E^{a,b}_k$ if and only if 
$E(l)^{p,q}_r$ degenerates at $E(l)^{a,b}_k$ for every $l$.
\end{enumerate}
\end{lemma}

\begin{proof}
The first two  properties are clear for $r=0$; by induction we may assume $r>0$ and 
items \eqref{it1.lem.ssdiquozienti}, \eqref{it2.lem.ssdiquozienti} true for $r-1$. We
have a commutative diagram
\[\xymatrix{E^{p-r+1}_{r-1}\ar[r]^{d}\ar[d]^{\pi'}&
E^p_{r-1}\ar[r]^d\ar[d]_{\pi}&E^{p+r-1}_{r-1}\ar[d]^{\pi''}\\
E(l)^{p-r+1}_{r-1}\ar[r]^{d}&
E(l)^p_{r-1}\ar[r]^d&E(l)^{p+r-1}_{r-1}}\]
If $p+r\le l$, by induction the maps $\pi',\pi$ are 
isomorphisms, $\pi''$ is injective and then also the induced map $E^{p}_r\to E(l)^{p}_r$ is an isomorphism.
If $p<l$ and $p+r>l$, then  $E(l)^{p+r-1}_{r-1}=0$, $\pi$ is injective and 
$\pi'$ is an isomorphism; thus the induced map $E^{p}_r\to E(l)^{p}_r$ is injective.

As regards \eqref{it3.lem.ssdiquozienti}, the if part follows immediately from 
\eqref{it2.lem.ssdiquozienti}. Conversely, if $E^{p,q}_r$ degenerates at 
$E^{a,b}_k$, then for every $l$ and every 
$r\ge k$ we have a commutative diagram:
\[ \xymatrix{E^{a,b}_r\ar[d]^{\pi'}\ar[r]^0&E^{a+r,b-r+1}_r\ar[d]_{\pi}\\
E(l)^{a,b}_r\ar[r]^{d_r}&E(l)^{a+r,b-r+1}_r}\]

If $a+r\ge b$ then $E(l)^{a+r,b-r+1}_r=0$, while if $a+r<b$, then 
$\pi'$ is surjective; in both cases $d_r=0$.
\end{proof}

Thus, putting together Theorem~\ref{thm.formalitycriteria} and Lemma~\ref{lem.ssdiquozienti} we obtain the following proposition.

\begin{proposition}\label{prop.formalitycriteriaquotient} 
Let $L$ be a differential graded Lie algebras. For every positive integer $l$ let
$\tau_{<l}E(L,L)^{p,q}_r$
be the spectral sequence of  the complex 
\[ \frac{CE(L,L)}{F^lCE(L,L)}=\prod_{p=0}^{l-1}\Hom^*_{\K}(L^{\wedge p},L)\;.\]
Then the following conditions are equivalent:

\begin{enumerate}

\item $L$ is formal;

\item the spectral sequence $\tau_{<l}E(L,L)^{p,q}_r$ degenerates at $E_2$ for every $l$;

\item the spectral sequence $\tau_{<l}E(L,L)^{p,q}_r$ degenerates at $E_2^{1,0}$ for every $l$;

\item\label{item.thm.formalitycriteria3}  denoting by 
$e\in \tau_{<l}E(L,L)^{1,0}_2$
the Euler class of $L$, 
we have $d_r(e)=0\in \tau_{<l}E(L,L)^{r+1,1-r}_r$ for every $r\ge 2$; and every $l$;
\end{enumerate}
\end{proposition}

\bigskip
\section{The role of formal DG-Lie algebras in deformation theory}
\label{sec.formalityrole}

The role of differential graded Lie algebras in deformation theory was clear since the mid sixties when Nijenhuis and Richardson \cite{NijRich64,NijRich67} observed that many deformation problems are controlled by the Maurer--Cartan equation
\[ dx+\frac{1}{2}[x,x]=0,\qquad x\in L^1,\]
for a suitable differential graded Lie algebra $L$. This point of view was extended  
 by Deligne  in terms of the philosophy that ``in 
characteristic 0 every (infinitesimal) deformation problem  is controlled by a differential graded Lie algebra, with quasi-isomorphic DG-Lie algebras giving the same deformation theory'', \cite{GoMil1}.

A more precise statement of the above philosophy can be stated in the framework of functors of Artin rings. 
Following \cite{K}, let $\K$ be a field of characteristic $0$, let $\Art$ be the category of local Artin $\K$-algebras with residue field $\K$ and let $F\colon \Art\to \Set$ be the  functor of infinitesimal deformations of some ``good'' 
algebro-geometric structure. Then there exists a differential graded Lie algebras $(L,d,[-,-])$ such that $F\simeq\Def_L$, where 
\[ \Def_L(A)=\frac{\{x\in L^1\otimes\mathfrak{m}_A\mid dx+\frac{1}{2}[x,x]=0\}}{
\text{gauge action of }\exp(L^0\otimes\mathfrak{m}_A)}\;,\]
$\mathfrak{m}_A$ is the maximal ideal of $A$ and the gauge action is defined by the formula
\[ e^a\ast
x:=x+\sum_{n=0}^{\infty}\frac{[a,-]^n}{(n+1)!}([a,x]-da),\quad  a\in L^0\otimes\mathfrak{m}_A,\;
x\in L^1\otimes\mathfrak{m}_A .\]%
The basic theorem of deformation theory asserts that if $f\colon L\to M$ is a quasi-isomorphism of differential graded Lie algebras, then the induced natural transformation $f\colon \Def_L\to \Def_M$ is an isomorphism (see \cite{ManettiSeattle} and reference therein).

\begin{proposition}\label{prop.secitsenough}
If a differential graded Lie algebra  $L$ is formal, then the two
maps
\[
\begin{array}{r}
\Def_L(\K[t]/(t^3))\to \Def_L(\K[t]/(t^2))\\
~\\
\Def_L(\K[[t]]):=\displaystyle\lim_{\leftarrow n} \Def_L(\K[t]/(t^n))\to \Def_L(\K[t]/(t^2))
\end{array}\]%
have the same image. 
\end{proposition}

\begin{proof}
Since $\Def_L$ is invariant under quasi-isomorphisms  
we may assume that $L$ has trivial differential  
and therefore its Maurer--Cartan equation becomes
$[x,x]=0$, $x\in L^1$.
Therefore $tx_1\in \Def_L(\K[t]/(t^2))$ lifts to
$\Def_L(\K[t]/(t^3))$ if and only if there exists $x_2\in L^1$
such that
\[ t^2[x_1,x_1]\equiv[tx_1+t^2x_2,tx_1+t^2x_2]\equiv 0\pmod{t^3}\iff
[x_1,x_1]=0\]%
and $[x_1,x_1]=0$ implies that $tx_1\in \Def_H(\K[t]/(t^n))$ for
every $n\ge 3$.  
\end{proof}

Notice that  the formality of $L$  does not imply that
$\Def_L(\K[[t]])\to \Def_L(\K[t]/(t^3))$ is surjective.
The reader can easily verify that for a generic graded vector space $V$ and 
$L=\Hom^{*}_{\K}(V,V)$ the map  
$\Def_L(\K[t]/(t^{n+1}))\to \Def_L(\K[t]/(t^{n}))$ is not surjective for every
$n\ge 2$. 

The proof of Proposition~\ref{prop.secitsenough} also implies that, 
when the deformation problem controlled by $L$ admits a local moduli space $\mathcal{M}$ and  
$L$ is formal, then $\mathcal{M}$ is defined 
by quadratic equations; more precisely $\mathcal{M}$ is isomorphic to the germ at $0$ of the quadratic cone
defined by the equation $[x,x]=0$, $x\in H^1(L)$; for a more detailed discussion and applications we refer to  \cite{GoMil1,MartPadova}.

Probably, the simplest example of local  moduli space which is not defined by quadratic equations is 
given by the  Hilbert scheme representing embedded deformations of the closed point inside the affine scheme 
$\operatorname{Spec}(\K[x]/(x^3))$. 
By  standard deformation theory, see e.g. \cite{semireg2011}, the construction of the differential graded Lie algebra $L$ controlling this deformation problem
is described by the following three steps:
\begin{enumerate}

\item replace the $\K$-algebra $\K[x]/(x^3)$ with a Koszul--Tate resolution, for instance with the DG-algebra
\[ R=(\K[x,y],d),\qquad \bar{x}=0,\quad \bar{y}=-1,\qquad d(y)=x^3\;,\] 
where the closed point is the subscheme defined by the differential ideal $I=(x,y)\subset R$.
 
\item consider the differential graded Lie algebra $M=\Der^*(R,R)$ and its subalgebra 
\[ N=\{\alpha\in M\mid \alpha(I)\subset I\}\;.\]

\item take $L$ as the homotopy fiber of  the inclusion $N\subset M$; as a concrete description 
of $L$ we can take
\[ L=\{ a(t)\in M[t,dt]\mid a(0)=0,\; a(1)\in N\}.\] 
 
\end{enumerate}

The non formality of $L$ can also be checked algebraically,  without relying 
on deformation theory. In fact 
$M$ is the free $\K[x]$-module generated by 
\[ y\frac{d~}{dx}\,,\quad   \frac{d~}{dx}\,,\, y\frac{d~}{dy}\,,\qquad \frac{d~}{dy}\;,\]
while $N$ is the $\K[x]$-submodule generated by 
\[ y\frac{d~}{dx}\,,\quad   x\frac{d~}{dx}\,,\, y\frac{d~}{dy}\,,\qquad x\frac{d~}{dy}\;.\]
There exists a direct sum decomposition $M=N\oplus A$ where $A=A^0\oplus A^1$ is the graded vector space 
generated by $u:=\frac{d~}{dx}$ and $\frac{d~}{dy}$. Since $A$ is an abelian graded Lie subalgebra of
$M$ and  $d=x^3\frac{d~}{dy}\in N$ we can apply Voronov's construction of higher derived brackets of an inner derivation \cite{bandieraderived,voronov2}. Denoting by $P\colon M\to A$ the projection with kernel $N$, the 
maps of degree $+1$:
\[ q_n\colon A^{\odot n}\to A,\qquad q_n(a_1,\ldots,a_n)=P[[\cdots[[d,a_1],a_2],\ldots],a_n],\qquad n\ge 1,\]
give an $L_{\infty}[1]$ structure on $A$ which, according to \cite[Thm. 1.3]{bandieraderived},  
is weak equivalent to the d\'ecalage of $L$.
Since
\[ [d,u]=\left[x^3\frac{d~}{dy},\frac{d~}{dx}\right]=-3x^2\frac{d~}{dy},\qquad
[[d,u],u]=\left[-3x^2\frac{d~}{dy},\frac{d~}{dx}\right]=6x\frac{d~}{dy},\]
\[ [[[d,u],u],u]=\left[6x\frac{d~}{dy},\frac{d~}{dx}\right]=6\frac{d~}{dy},\] 
we have $q_1=q_2=0$ and $q_3\not=0$. Thus the $L_{\infty}[1]$-algebra $(A,q_1,q_2,q_3,\ldots)$ 
is not formal.

\medskip
\textbf{Acknowledgements.} I'm indebted with  Enrico Arbarello for useful discussions and for pointing my attention on the  paper~\cite{kaledin}. I also would like to thanks the referee for an useful comment and for pointing out some minor mistakes.

\appendix
\bigskip
\section{Erratum and corrigenda}

Coline Emprin kindly informed us of a mistake in the proof of Lemma~\ref{lem.criteriodicollassamentoparziale}. 
In fact, the equality $R_2\widehat{\alpha}-(-1)^{\overline{\alpha}}\widehat{\alpha}Q_2=0$  holds whenever the $L_{\infty}$-morphism $f$ is linear, i.e., 
all its Taylor components of higher degree vanish.

A corrected version of Lemma~\ref{lem.criteriodicollassamentoparziale} is the following result.

\begin{lemma}[Corrigendum of Lemma~\ref{lem.criteriodicollassamentoparziale}]\label{lem.corrected5.3}
Let $(E(V,W;f)_r,d_r)$ be the Chevalley--Eilenberg spectral sequence of 
a linear $L_{\infty}$-morphism  
$f\colon (V,0,q_2,q_3,\ldots)\dashrightarrow (W,0,r_2,r_3,\ldots)$  of minimal $L_{\infty}[1]$-algebras. Assume that for some integer $k\ge 3$ we have 
$q_3=\cdots=q_k=0$ and 
$r_3=\cdots=r_k=0$. Then $d_r=0$ for every $2\le r<k$ and therefore
\[ E(V,W;f)_2=E(V,W;f)_3=\cdots=E(V,W;f)_{k}.\]
\end{lemma}

\begin{proof} Take the same proof of Lemma~\ref{lem.criteriodicollassamentoparziale} with the additional assumption of the linearity of $f$.
\end{proof} 

This weaker version does not affect 
the results from Lemma~\ref{lem.formalityimpliesdegeneration} to Corollary~\ref{cor.abelianitycriterion}, since in those cases Lemma~\ref{lem.criteriodicollassamentoparziale} is always applied for $f$ equal to the identity morphism. In particular, the main Theorem~\ref{thm.formalityminilamlLinfinity1} and its consequence Theorem~\ref{thm.formalitycriteria} are still valid, as well as all the results in Sections~\ref{sec.formalityDGLA} 
and~\ref{sec.formalityrole}. 

Instead, the proofs of the formality transfer theorems~\ref{thm.formalitytransferinfinity}
and \ref{thm.formalitytransfer} are no longer correct, and \cite[Example 4.5]{MR26} shows that the their statements are false. Fortunately, the following slightly weaker versions are still sufficient for proving 
Corollaries~\ref{cor.prodottoLie} and \ref{cor.tensorformality}, as well as  for most concrete applications, e.g. \cite{BMM2}.

\begin{theorem}[Corrigendum of Theorem~\ref{thm.formalitytransferinfinity}]
\label{thm.formalitytransferinfinity2} 
Let $f\colon V\dashrightarrow W$ be an $L_{\infty}$-morphism of $L_{\infty}[1]$-algebras such that:
\begin{enumerate}

\item $W$ is formal;

\item the map $f_*\colon E(V,V)^{p+1,-p}_p\to E(V,W;f)^{p+1,-p}_p$ is injective for every 
$p\ge 2$.
\end{enumerate}
Then also $V$ is formal.
\end{theorem}

\begin{proof}  
We have two morphisms of spectral sequences 
\[ \xymatrix{E(V,V)^{p,q}_r\ar[r]^-{f_*}&E(V,W;f)^{p,q}_r&E(W,W)^{p,q}_r\ar[l]_-{f^*}}\]
Denoting by $e_V,e_W$ and $e_f$ the Euler classes of $V,W$ and $f$ respectively we have 
$f_*(e_V)=e_f=f^*(e_W)$.
According to Corollary~\ref{cor.formalityLinfinity} we have $d_p(e_W)=0$ for every $p\ge 2$ and therefore
also $d_p(e_f)=d_p(f^*(e_W))=f^*(d_pe_W)=0\in E(V,W;f)^{p+1,-p}_p$ for every $p$.
The injectivity of $f_*\colon E(V,V)^{p+1,-p}_p\to E(V,W;f)^{p+1,-p}_p$ implies also 
$d_pe_V=0$ for every $p\ge 2$ and therefore also $V$ is formal.  
\end{proof}

The above theorem, applied to morphisms of DG-Lie algebras, gives the following result.

\begin{theorem}[Corrigendum of Theorem~\ref{thm.formalitytransfer}]\label{thm.formalitytransfer2} 
Let $f\colon L\to M$ be a morphism of 
differential graded Lie algebras. Assume that 
\begin{enumerate}

\item  $M$ is formal;

\item   for every $p\ge 2$ the map 
\[ f\colon E(L,L)^{p+1,1-p}_p\to E(L,M)^{p+1,1-p}_p\]
is injective.
\end{enumerate}
Then also $L$ is formal.
\end{theorem}

\begin{proof}  If $V=L[1]$ and $W=M[1]$ are the  $L_{\infty}[1]$-algebras associated to $L$ and $M$, the degree shifting gives isomorphisms of spectral sequences 
\[ E(V,V)_r^{p,q}\simeq E(L,L)_r^{p,q+1},\qquad 
E(V,W;f)_r^{p,q}\simeq E(L,M)_r^{p,q+1},\]
and the conclusion follows from Theorem~\ref{thm.formalitytransferinfinity2}.
\end{proof}

Notice that Theorem~\ref{thm.formalitytransfer2} still implies that  a graded Lie algebra $H$ with $E(H,H)^{p+1,1-p}_2=0$ for every $p\ge 2$ is intrinsically formal. In fact, if $L$ is a DG-Lie algebra such that $H^*(L)=H$, then $E(L,L)_2=E(H,H)_2$ and therefore $E(L,L)^{p+1,1-p}_p=0$ for every $p\ge 2$; it is now sufficient to consider the trivial morphism  $f\colon L\to 0$.

\begin{remark}
Condition 2 of Theorem~\ref{thm.formalitytransfer2} is satisfied in particular when 
$f$ is injective and there exists a decomposition $M=f(L)\oplus N$, where
$N$ is a differential graded $f(L)$-submodule of $M$, cf. \cite[Theorem 1.1]{MR26}.  

Similarly, Condition 2 of Theorem~\ref{thm.formalitytransfer} is satisfied in particular when 
$f$ is injective in cohomology and there exists a decomposition $H^*(M)=H^*(L)\oplus N$, where
$N$ is a graded $H^*(L)$-submodule of $H^*(M)$, cf. \cite[Example 4.5]{MR26}
\end{remark}  

\medskip
\textbf{Acknowledgements.} I'm indebted with Coline Emprin for pointing out the error in the proof of Lemma~\ref{lem.criteriodicollassamentoparziale}.

\def\refname{Bibliography of the appendix}


\bigskip
\begin{thebibliography}{99}

\raggedright

\bibitem{bandieraderived} R. Bandiera: 
\newblock \emph{Non-abelian higher derived brackets.}
J. Pure Appl. Algebra (2014), doi:10.1016/j.jpaa.2014.10.015,
\newblock \texttt{arXiv:1304.4097 [math.QA]}.


\bibitem{bandierakapra} R. Bandiera:  
\emph{Formality of Kapranov's brackets on pre-Lie algebras.}
\texttt{arXiv:1307.8066 [math.QA]}.


\bibitem{BaMacoisotropic} R. Bandiera and M. Manetti: 
\emph{On coisotropic deformations of holomorphic submanifolds.} 
J. Math. Sci. Univ. Tokyo {\bf 22} (2015), no.~1, 1--37; \texttt{arXiv:1301.6000v2}.


\bibitem{BMM2} R. Bandiera, M. Manetti and F. Meazzini: \emph{Formality conjecture for minimal surfaces of Kodaira dimension 0.} Compos. Math. {\bf 157} (2021), no.~2, 215--235.



\bibitem{berglund} A. Berglund: \emph{Rational homotopy theory of mapping spaces via Lie theory for 
$L_{\infty}$-algebras.} arXiv:1110.6145v1 [math.AT].

\bibitem{bumu} U. Buijs and A. Murillo: 
\emph{Algebraic models of non-connected spaces and homotopy theory of $L_{\infty}$-algebras.} Adv. Math., \textbf{236} (2013), 60-91.


\bibitem{CalRos} D. Calaque and C. A. Rossi: 
\emph{Lectures on Duflo isomorphisms in Lie algebra and complex geometry.}
Ems Series of Lectures in Mathematics,
European Mathematical Society (2011).


\bibitem{CE} C. Chevalley and S. Eilenberg:
\emph{Cohomology Theory of Lie Groups and Lie Algebras.} 
Transactions of the American Mathematical Society, Vol. \textbf{63}, No. 1 (1948), 85-124.

\bibitem{DGMS} P.~Deligne, P.~Griffiths, J.~Morgan and D.~Sullivan:
\emph{Real homotopy theory of K\"ahler manifolds.} Invent. Math. 
\textbf{29} (1975) 245-274.


\bibitem{FMpoisson} D. Fiorenza and M. Manetti:
\emph{Formality of Koszul brackets and deformations of holomorphic Poisson manifolds.}
Homology, Homotopy and Applications, \textbf{14}, No. 2, (2012), 63-75;
\texttt{arXiv:1109.4309v3 [math.QA]}.



\bibitem{Gode} R.~Godement:
\emph{Topologie alg\'ebrique et th\'eorie des faisceaux.}
Hermann, Paris (1958).



\bibitem{GoMil1} W.M. Goldman and J.J. Millson:
\emph{The deformation theory of
representations of fundamental groups of compact K\"{a}hler manifolds}
Publ. Math. I.H.E.S. \textbf{67} (1988) 43-96.



\bibitem{halsta} S. Halperin and J. Stasheff: \emph{Obstructions to homotopy equivalences.} 
Adv. in Math. No. 32, 233-279 (1979).



\bibitem{hinichDGC}
V.~Hinich: \emph{DG coalgebras as formal stacks.}
J. Pure Appl. Algebra \textbf{162} (2001) 209-250;
\texttt{arXiv:math/9812034v1 [math.AG]}.


\bibitem{hinich}
V. Hinich: \emph{Tamarkin's proof of Kontsevich formality theorem.} 
Forum Math. 15 (2003), no. 4, 591-614.
\texttt{ArXiv:math/0003052}. 





\bibitem{algebraicBTT} D. Iacono and M. Manetti: 
\emph{An algebraic proof of Bogomolov-Tian-Todorov theorem.} 
In Deformation Spaces vol. \textbf{39}, Vieweg Verlag (2010), 113-133;  
{\texttt{arXiv:0902.0732}}. 


\bibitem{semireg2011} D. Iacono and M. Manetti:
\emph{Semiregularity and obstructions of complete intersections.}
Advances in Mathematics \textbf{235} (2013) 92-125;
\texttt{arXiv:1112.0425 [math.AG]}.

\bibitem{IaconoDP} D. Iacono: \emph{Deformations and obstructions of pairs $(X,D)$}.
\texttt{arXiv:1302.1149v5 [math.AG]}.


\bibitem{Ja} N. Jacobson: \emph{Lie algebras.} Wiley \& Sons (1962).

\bibitem{kaledin} D. Kaledin: \emph{Some remarks on formality in families.} 
Mosc. Math. J. \textbf{7} (2007) 643-652; 
\texttt{arXiv:math/0509699v4 [math.AG]}.


\bibitem{KKP} L. Katzarkov, M. Kontsevich and T. Pantev:
\emph{Hodge theoretic aspects of mirror symmetry.}  
From Hodge theory to integrability and TQFT tt*-geometry, 87-174,
Proc. Sympos. Pure Math., \textbf{78}, Amer. Math. Soc., Providence, (2008).
\texttt{arXiv:0806.0107v1 [math.AG]}


\bibitem{K} M. Kontsevich:
\emph{Deformation quantization of Poisson manifolds, I.} Letters
in Mathematical Physics~\textbf{66} (2003) 157-216;
{\texttt{arXiv:q-alg/9709040}}.


\bibitem{kontsevichSoibelmanbook} M. Kontsevich and Y. Soibelman: \emph{Deformation theory, I}. 
Book in progress. Available from Y. Soibelman's web page.

\bibitem{LadaMarkl} T.~Lada and M.~Markl: \emph{Strongly homotopy Lie
algebras.} Comm. Algebra \textbf{23} (1995) 2147-2161; \texttt{hep-th/9406095}.


\bibitem{laza} A. Lazarev. \emph{Maurer-Cartan moduli and models for function spaces.} Adv. Math., \textbf{235} (2013) 296-320.

\bibitem{lazarev} A. Lazarev: \emph{Models for classifying spaces and derived deformation theory.} 
Proc. Lond. Math. Soc. (3) \textbf{109} (2014), no. 1, 40-64. 
\texttt{arXiv:1209.3866 [math.AT]}.



\bibitem{lunts} V. Lunts: 
\emph{Formality of DG algebras (after Kaledin).}
Journal of Algebra, Volume \textbf{323}, Issue 4, (2010), 878-898.
\texttt{arXiv:0712.0996 [math.AG]}.


\bibitem{ManRendiconti} M. Manetti:
\emph{Lectures on deformations of complex manifolds.} Rend. Mat.
Appl. (7) \textbf{24} (2004) 1-183;
\texttt{arXiv:math.AG/0507286}.


\bibitem{ManettiSeattle}
M. Manetti: \emph{Differential graded Lie algebras and  formal
deformation theory.} In \emph{Algebraic Geometry: Seattle 2005.}
Proc. Sympos. Pure Math. \textbf{80} (2009) 785-810.



\bibitem{indam} 
M. Manetti: \emph{Kodaira-Spencer formality of products of complex manifolds.} 
Trends in Contemporary Mathematics, Springer INdAM Series \textbf{8} (2014) 85-95.



\bibitem{MartPadova} E. Martinengo: 
\emph{Local structure of Brill-Noether strata in the moduli space of flat stable bundles.} 
Rend. Semin. Mat. Univ. Padova \textbf{121} (2009), 259-280. \texttt{arXiv:0806.2056 [math.AG].}



\bibitem{NijRich64} A. Nijenhuis and R. W. Richardson: \emph{
Cohomology and deformations of algebraic structures.}
Bull. Amer. Math. Soc. Volume \textbf{70}, Number 3 (1964), 406-411. 


\bibitem{NijRich67} A. Nijenhuis and R. W. Richardson: 
\emph{Deformation of Lie algebra structures.}
J. Math. Mech. \textbf{17} (1967), 89-105.



\bibitem{Pen}  M. Penkava: \emph{L-infinity algebras and their cohomology.}
\texttt{arXiv:q-alg/9512014}.



\bibitem{Qui} D.~Quillen: \emph{Rational homotopy theory.}
Ann. of Math. \textbf{90} (1969) 205-295.


\bibitem{voronov2} Th. Voronov:
\emph{Higher derived brackets for arbitrary derivations.} 
Travaux math\'ematiques, fasc. \textbf{XVI},
Univ. Luxemb., Luxembourg (2005), 163-186; \texttt{arXiv:0412202 [math.QA]}.


\end{thebibliography}

\begin{thebibliography}{9}
\setcounter{enumiv}{36} 
\bibitem{MR26} M. Manetti and G. Rossetti: \emph{Some examples of DG-Lie formality transfer.}
Preprint (2026); \texttt{arXiv:2601.19698v3}.

\end{thebibliography}
\end{document}